\RequirePackage{fix-cm}
\documentclass{article}
\usepackage{arxiv}
\usepackage{graphicx}
\usepackage{svg}
\usepackage{url}
\usepackage{amsmath}
\usepackage{amsfonts}
\usepackage{amssymb}
\usepackage{amsopn}
\usepackage{amsthm}
\usepackage{algorithm}
\usepackage{color}
\usepackage{algorithmic}
\usepackage{cite}
\usepackage{authblk}
\usepackage{hyperref}
\hypersetup{
}
\usepackage{mydefs}
\usepackage{tikz-cd}

\usepackage{cleveref}

\makeatletter
\newcommand{\ALC@uniqueautorefname}{Line}
\makeatother

\theoremstyle{plain}
\newtheorem{theorem}{Theorem}[section]
\newtheorem{proposition}[theorem]{Proposition}
\newtheorem{lemma}[theorem]{Lemma}
\newtheorem{corollary}[theorem]{Corollary}

\theoremstyle{definition}
\newtheorem{definition}[theorem]{Definition}

\theoremstyle{remark}
\newtheorem{remark}[theorem]{Remark}
\newtheorem{comment}[theorem]{Comment}

\begin{document}

\title{Parameterized Vietoris-Rips Filtrations via Covers
\thanks{BJN was supported by the Defense Advanced Research Projects Agency (DARPA) under Agreement No.
HR00112190040}
}


\author{Bradley J. Nelson}

\author{
Bradley J. Nelson\\
	Department of Statistics\\
	Committee on Computational and Applied Mathematics\\
	University of Chicago\\
	Chicago, IL 60637\\
	\texttt{bradnelson@uchicago.edu} \\
}



\date{Received: date / Accepted: date}

\maketitle

\begin{abstract}
A challenge in computational topology is to deal with large filtered geometric complexes built from point cloud data such as Vietoris-Rips filtrations. This has led to the development of schemes for parallel computation and compression which restrict simplices to lie in open sets in a cover of the data.  We extend the method of acyclic carriers to the setting of persistent homology to give detailed bounds on the relationship between Vietoris-Rips filtrations restricted to covers and the full construction.  We show how these complexes can be used to study data over a base space and use our results to guide the selection of covers of data.  We demonstrate these techniques on a variety of covers, and show the utility of this construction in investigating higher-order homology of a model of high-dimensional image patches.


\end{abstract}




\section{Introduction}\label{sec:introduction}

A common task in computational topology is to construct a (filtered) geometric complex from a set of points $\bX$, possibly sampled from some larger space $X\supseteq \bX$, using a pairwise dissimilarity $d:\bX \times \bX \to \RR_+$ between points.  Two major applications include statistical recovery of homological features of the larger space $X$, 
\cite{CImgPatch,carlsson_topological_2014}
perhaps in the process of exploratory data analysis, and generating features for machine learning tasks \cite{cang_topologynet:_2017,hiraoka_hierarchical_2016}.
One limitation of geometric constructions is that they can produce very large combinatorial representations of a space as simplicial complexes, typically growing in the number of points $n$ and maximal simplex dimension $q$ as  $O(n^{q+1})$ total simplices.  Another limitation is that one must consider the choice of dissimilarity $d$.  In general, a dissimilarity may be trusted locally (for small values), but not globally (for large values) -- a key motivation for dimension reduction techniques such as locally linear embeddings \cite{roweis_nonlinear_2000} and ISOMAP \cite{tenenbaum_global_2000}.

For example, if the points $\bX$ are sampled near a low dimensional manifold embedded in Euclidean space, we may choose the metric $d$ to either be the Euclidean distance of the ambient space, or the intrinsic distance of the manifold, perhaps approximated from the sampling.  At small distances, the choice of metric will not appear to matter much, but at large distances differences between the two metrics will become much more apparent.  These two factors combine to make the calculation of persistent homology from samples difficult even in dimensions as small as 2 or 3 -- either a large number of samples are required to cover a space without growing distance too large, or we must use large non-local distances which are not trusted.  

One way to make calculation of higher-dimensional homology of sampled point clouds tractable is to incorporate the additional structure of a map $f:X\to B$.  In this setting, the space $X$ is said to be parameterized by $B$, which is called the base space.  A variety of tools in continuous topology have developed, both in the context of homotopy theory which studies notions such as base-space preserving maps\cite{mayParametrizedHomotopyTheory2006} and fibrations \cite{serreHomologieSinguliereEspaces1951}, and in the context of homology where the Leray and Leray-Serre spectral sequences can be used to ease calculation \cite{McClearySS}.  Many ideas and results in the continuous setting rely on an analysis of {\em fibers} of the map, $f^{-1}(b)$, which poses a difficulty in the discrete setting where fibers will generally be empty.  In this paper, we consider an extension of parameterized spaces to the setting of filtered complexes based not on fibers but on inverse images of sets $f^{-1}(U)$.  Generally, the map $f$ is not needed for the construction - we can simply take any cover of the data (which coincides with $B = X$ and $f$ as the identity):
\begin{definition}\label{def:cover_system}
A {\em system of complexes} over a cover $\calU$ is a collection of (filtered) cell complexes $\{\X^T(U)\}_{U\in \calU}$ where $\X^T(U)$ has $U$ as its 0-skeleton, and the restriction of complexes to intersections of sets in the cover are compatible.
\begin{equation}
    \X^T(U_{i})|_{\cap U_k} =  \X^T(U_j)|_{\cap U_k}
\end{equation}
for all $U_i, U_j\in \{U_k\} \subseteq \calU$.
\end{definition}
\begin{definition}\label{def:cover_complex}
A {\em cover complex} $\X^T(\calU)$ is the union of complexes in a system of complexes.
\begin{equation}
    \X^T(\calU) = \bigcup_{U\in \calU} \X^T(U)
\end{equation}
\end{definition}

This definition of cover complex coincides with a similar definition which appeared in an early pre-print of \cite{ApproximateNerveTheorem2017}, but which was abandoned in subsequent versions. The goal of \cite{ApproximateNerveTheorem2017}, as well as associated literature \cite{chazalPersistencebasedReconstructionEuclidean2008, cavannaGeneralizedPersistentNerve2018} is to understand when a filtered nerve can effectively be used to approximate a larger computation, a question which we will address for cover complexes in \cref{sec:gen_nerve_theorem}.  In contrast, we will seek to use the actual cover complex in computations in situations where the complex restricted to each set is not necessarily close to acylic, which we will investigate in \cref{sec:cover_local_interleavings} and \cref{sec:cover_full_interleaving}.  This has previously been investigated by Yoon \cite{yoon2018} in the calculation of persistent homology of Vietoris-Rips filtrations at small scales in the setting where the nerve of the cover is contractible. These complexes also contain similarities to the multiscale mapper construction \cite{deyMultiscaleMapperTopological2016}, which also uses inverse images of sets in covers, but applies this to simplicial complexes generated using the mapper algorithm \cite{mapper} which contracts connected components in the inverse image of sets.  We shall be interested in higher-dimensional homology as well.




\subsection{Geometric Complexes}\label{sec:geometric_complexes}

In applied topology, there are a variety of constructions which allow for the construction of simplicial complexes from a data set $\bX$.  These complexes allow for the approximation of a larger space from which the data was sampled.  Common examples include the Vietoris-Rips complex, \v{C}ech complex, Witness complex, and others -- see \cite{geometric_stab2014} for a review of a variety of constructions.

In this paper, we will focus on Vietoris-Rips complexes which are attractive from a computational point of view because they allow for an easy combinatorial description in arbitrary dimensions (as opposed to \v{C}ech or $\alpha$-complexes), and do not require selection of landmarks as in Witness complexes.  The Vietoris-Rips complex uses a dissimilarity $d:\bX \times \bX \to \RR$ to determine whether simplices should be included in the complex.

\begin{definition}\label{def:ripsd}
Let $(\bX, d)$ be a dissimilarity space. We extend the dissimilarity to tuples of points $x_0,\dots,x_k\subseteq \bX$ as
\begin{equation}
d(x_0,\dots,x_k) = \max_{0\le i < j \le k} d(x_i, x_j)
\end{equation}
\end{definition}
with $d(x) = d(x,x) = 0$.

\begin{definition}
Let $(\bX, d)$ be a dissimilarity space. The Vetoris-Rips complex $\calR(\bX; r)$ is the union of simplices
\begin{equation}
    \calR(\bX; r) = \{(x_0,\dots,x_k) \mid x_0,\dots,x_k\in \bX, d(x_0,\dots,x_k) \le r\}.
\end{equation}
We can use the same notation to refer to a filtration by letting the $r$ parameter vary.
\end{definition}
Because the Rips filtration is a flag filtration, the simplex $(x_0,\dots,x_k)$ appears at parameter $d(x_0,\dots,x_k)$.  We can restrict simplicies of this full complex to sets in a cover to obtain an equivalent notion of cover complex:

\begin{definition}\label{def:cover_complex2}
Let $\X^T$ be a filtered cell complex over a poset $T$, with vertex set $\X^T_0 = X$, and let $\calU$ be a cover of $X$.  We define the cover complex $\X^T(\calU)$ to be the restriction of $\X^T$ to cells whose 0-skeleton lies in some $U\in \calU$.
\end{definition}
This definition agrees with \cref{def:cover_complex} where the system of complexes comes from the restriction of the full filtered complex $\X^T$ to sets in $\calU$.  In \cref{sec:application_to_rips} we will specifically consider Vietoris-Rips cover complexes, which we will denote $\calR(\bX, \calU; r)$.

\subsection{Homology, Persistence, and Interleavings}\label{sec:persistent_homology}

We are primarily interested in obtaining the persistent homology of filtered complexes, which can be used to describe the robust topological features in a filtration.  For additional background on homology, we recommend \cite{HatcherAT}, and for additional information on persistent homology and interleavings, we recommend \cite{Oudot}.
Given a filtration $\X^T$, the homology functor in dimension $q$ produces a persistence vector space  $H_q(\X^T)$, where for every filtration value $t\in T$ the complex $\X^t$ has an associated vector space $H_q(\X^t)$, and the inclusion maps $\X^s \subseteq \X^t$ for $s\le t$ have associated linear maps $F_q^{s,t}:H_q(\X^s) \to H_q(\X^t)$, as illustrated by the diagram:
\begin{equation}
\begin{tikzcd}
\X^s\ar[d] \ar[r,hookrightarrow] &\X^t\ar[d]\\
H_q(\X^s) \ar[r,"F_q^{s,t}"] &H_q(\X^t)
\end{tikzcd}
\end{equation}
The dimension, $\dim H_q(\X^t)$, can generally be interpreted to count the number of $q$-dimensional ``holes'' in the space $\X^t$, and the induced maps describe how holes relate to one another throughout the filtration.  We will generally consider our posets $T$ to be finite subsets of the real numbers $\RR$, for example, the critical values at which simplices appear in a Vietoris-Rips filtration.  In this case, the persistence vector space $H_q(\X^T)$ is described up to isomorphism by a collection of interval indecomposables $\{(b_i,d_i)\}$, or persistence barcode, which track the appearance (birth) and disappearance (death) of new homological features throughout the filtration \cite{ZCComputingPH2005,ZZtheory2010}.
In the context of geometric filtrations, intervals with long lengths $|d_i - b_i|$ are typically considered robust topological features, and those with short lengths are typically considered topological noise.

We wish to be able to compare the persistent homology of different filtrations, which is accomplished through the use of interleavings (cite).  We can consider persistence vector spaces abstractly as quiver representations \cite{gabrielI,ZZtheory2010} over the poset $T$, which we denote $V^T$ (forgetting that the vector spaces and linear maps came from homology).  In order to compare two different persistence vector spaces, we must first have a notion of map between them.
\begin{definition}\label{def:graded_map}
Let $V^S$ and $W^T$ be persistence vector spaces, and $\alpha:S\to T$ be a non-decreasing map.  An $\alpha$-shift map is a collection of linear maps $F^{\alpha} = \{F^{s}: V^s \to W^{\alpha(s)}\}_{s\in S}$ which commute with the maps in $V^S$ and $W^T$
\begin{equation}
\begin{tikzcd}
V^r \ar[r]\ar[d,"F^{r}"] & V^s \ar[d,"F^{s}"]\\
W^{\alpha(r)} \ar[r] & W^{\alpha(s)}
\end{tikzcd}
\end{equation}
\end{definition}
We denote the self-shift map $I^\alpha:V^S \to V^S$ as the map that simply follows the maps in the persistence vector space $I^\alpha:V^s \to V^{\alpha(s)}$.

An interleaving is a pair of shift maps between persistence vector spaces:
\begin{definition}\label{def:alpha_beta_interleaving}
An {\em $(\alpha,\beta)$-interleaving} between $V^S$ and $W^T$ is a pair of graded maps $F^\alpha: V^S\to W^T, G^\beta:W^T \to V^S$ so so that $G^\beta \circ F^\alpha \cong I^{\beta \circ \alpha}$ and $F^\alpha \circ G^\beta \cong I^{\alpha \circ \beta}$.
\end{definition}
If two persistence vector spaces are $(\alpha,\beta)$ interleaved, then any vector $v\in V^s$ with image in $V^{\beta \circ \alpha(s)}$ must have a non-zero image in $W^{\alpha(s)}$.  This provides a way to compare interval indecomposables in the context of persistent homology.

The interleaving distance \cite{chazalProximityPersistenceModules2009} is a distance on persistence vector spaces constructed by considering shift maps of the form $\epsilon: t \to t + \epsilon$.  The infimum over $\epsilon \ge 0$ that admits an $(\epsilon,\epsilon)$ interleaving of two persistence vector spaces is the in
\begin{equation}
    d_I(V^S, W^T) = \inf \{ \epsilon \ge 0 \mid \exists (\epsilon,\epsilon) \text{ interleaving of } V^S, W^T \}
\end{equation}
In the case where more general shift maps $\alpha, \beta \ge \epsilon$, then an $(\alpha,\beta)$-interleaving bounds the interleaving distance between persistence vector spaces from above.  In the case of single-parameter persistence, the interleaving distance is equivalent to the bottleneck distance on persistence diagrams \cite{lesnick_multid2015}.

Interleavings are often used to obtain stability results explaining how perturbations of an input can affect output persistence vector spaces. An early use application of interleavings was to Gromov-Hausdorff stability of the persistent homology of Vietoris-Rips filtrations.
\begin{theorem}\label{thm:gh_stability}
\cite{GHStable,geometric_stab2014}
Let $(\bX, d_X)$ and $(\bY, d_Y)$ be metric spaces with 
$$d_{GH}((\bX, d_X), (\bY,d_Y)) \le \epsilon.$$
Then $H_q(\calR((\bX,d_X); r))$ and $H_q(\calR((\bY,d_Y); r))$ are $(\epsilon,\epsilon)$-interleaved.
\end{theorem}

\subsection{Outline/Contributions}\label{sec:outline}



In this paper, we develop the use of Vietoris-Rips cover complexes, $\calR(\bX, \calU; r)$, with an eye to understanding homological stability properties and their relationship to the full Vietoris-Rips construction.  In \cref{sec:filtered_carriers_and_interleavings} we develop a filtered version of the acyclic carrier theorem which can be used to construct interleavings from initial data.  In section \cref{sec:covers}, we build up local-to-global results including Hausdorff stability of $H_q$ and a generalized Nerve theorem.  In section \cref{sec:application_to_rips} we characterize the relationship between $H_q(\calR(\bX; r))$ and $H_q(\calR(\bX, \calU; r))$ in terms of interleavings.  Finally, in \cref{sec:computations} we demonstrate the use of Vietoris-Rips cover complexes over base spaces, and target the computation of high-dimensional homology groups of a fiber-bundle associated to high-dimensional image patches.  Several of these results were presented in preliminary form in the dissertation of the author \cite{nelson_parameterized_2020}.  The present paper includes a simplified and focused exposition, new results relating Vietoris-Rips cover complexes to sparse filtrations, and additional computational examples.


\section{Filtered Carriers and Interleavings}\label{sec:filtered_carriers_and_interleavings}


In this section, we introduce a notion of filtered carrier between complexes, and use this to construct explicit interleavings between persistence vector spaces.  This generalizes the definition of carriers used in algebraic topology.  Historically, carriers were used to prove equivalence of various homology theories -- see \cite{eilenbergSteenrod1952,MunkresAT,MosherTangora} for additional background.

\subsection{Filtered Maps and Carriers}\label{sec:filtered_carriers}

We define filtered carriers for objects in a category filtered by partially-ordered sets (posets) $S,T$ with initial objects.  For our purposes, we consider totally ordered $S,T \subseteq \RR_+$ (with initial object 0), but extensions to other partially ordered sets are possible, with additional conditions, which allow for applications to generalized or multiparameter persistence.  In order to specialize these results to standard carriers in the non-filtered setting, it suffices to consider the single element poset $S = T = \{0\}$.

\begin{definition}\label{def:filtered_object}
A filtered object in a category over a poset $T$ is a collection of objects $\X^T = \{\X^t\}_{t\in T}$ where $\X^{t_1} \subseteq \X^{t_2}$ if $t_1\le t_2$.
\end{definition}
The types of filtered objects we will consider are filtered cell complexes and filtered chain complexes.

\begin{definition}\label{def:filtered_map}
Let $\X^S, \Y^T$ be filtered objects in a category over posets $S, T$ respectively.  Let $\alpha:S \to T$ be a non-decreasing map.  An $\alpha$-shift map $f^\alpha:\X^S\to \Y^T$ is a collection of maps $f^{s}:\X^s \to \Y^{\alpha(s)}$ for each $s\in S$ so that the following diagram commutes.
\begin{equation}
\begin{tikzcd}
\X^s \ar[r]\ar[d,"f^s"] &\X^{s'}\ar[d,"f^{s'}"]\\
\Y^{\alpha(s)} \ar[r] &\Y^{\alpha(s')}
\end{tikzcd}
\end{equation}
\end{definition}
We are primarily interested in the categories of cell complexes and chain complexes.  If $\alpha, \beta: S\to T$ are non-decreasing maps and $\alpha(s) \le \beta(s)$ for all $s\in S$, then we can extend a filtered map $f^\alpha$ to a filtered map $f^\beta$ by first applying $f^\alpha$ and then shifting the filtration to $\beta$: $f^\beta = \iota^{\beta - \alpha} \circ f^{\alpha} $. While the above definition can be applied to homotopies as well, we want to give a specialized definition of a sort of filtered chain homotopy:

\begin{definition}\label{def:filtered_htpy}
Let $F^\alpha_\ast, G^\alpha_\ast :C_\ast^S \to D_\ast^T$ be $\alpha$-shift maps of chain complexes.  We say $F^\alpha, G^\alpha$ are $\beta$-chain homotopic, where $\beta:T\to T$ is a non-decreasing map if there exists a collection of maps $K^s_q : C_q^s \to D_{q+1}^{\beta\circ \alpha(s)}$ $q = 0,1,\dots$, and $s\in S$, so that
\begin{equation}
\partial^D_{q+1}K_q^s + K_{q-1}^s\partial^C_q = \iota^\beta (G_q^s - F_q^s)
\end{equation}
\end{definition}

\begin{definition}\label{def:filtered_carrier}
A filtered carrier of chain complexes over a poset $T$, denoted $\scrC^T: C_\ast^S \to D_\ast^T$ is an assignment of basis vectors of $C_\ast^S$ to filtered sub-complexes of $D_\ast^T$.  In situations where $T$ is understood, we will drop the superscript, and simply write $\scrC: C_\ast^S \to D^T_\ast$.
\end{definition}
Note that while a basis element $x\in C_\ast^S$ may appear at parameter $s\in S$, the carrier $\scrC^T(x)$ is filtered by $T$. We can also define a filtered carrier of cell complexes $\scrC^T:\X^S \to \Y^T$ by assigning cells of $\X^S$ to sub-cell complexes of $\Y^T$.  A (filtered) carrier of cell complexes produces a (filtered) carrier of chain complexes by application of the cellular chain functor.

We say the carrier $\scrC$ is \emph{proper} with respect to the filtered bases $B_\ast^S$ of $C_\ast^S$ and $B_\ast^T$ of $D_\ast^T$ if $\scrC(x)$ is generated by a sub-basis of $B_\ast^T$ for each $x$ in the basis $B_\ast^S$.  Note that carriers of cell complexes always produce carriers of chain complexes that are proper with respect to the cell basis. 

The term ``carrier'' comes from the utility of carrying a map:

\begin{definition}\label{def:carry_filtered_map}
Let $\scrC^T:C^S_\ast \to D^T_\ast$ be a filtered carrier, and $F^\alpha_\ast$ be an $\alpha$-shift chain map.  We say that $F^\alpha_\ast:C^S_\ast \to D^T_\ast$ is carried by $\scrC^T$ if $F^\alpha(x) \in \scrC^T(x)$ at parameter $\alpha(s)$ for all basis elements $x\in C^s_\ast$.
\end{definition}
Again, there is an analogous definition for carriers of filtered cell complexes and maps.

\subsection{A Filtered Acyclic Carrier Theorem}

Recall that a chain complex $C_\ast$ is acyclic if its reduced homology $\tilde{H}_q(C_\ast) = 0$ for all $q\ge 0$.  A carrier of chain complexes $\scrC:C_\ast \to D_\ast$ is acyclic if $\scrC(x)$ is acyclic for all basis elements $x\in C_\ast$.
The primary utility of acyclic carriers is in providing a tool to extend maps from initial data. For ordinary (non-filtered) chain complexes, we have
\begin{theorem} \label{thm:acyclic_carrier}
(Acyclic carrier theorem) If $\scrC: C_\ast \to D_\ast$ is acyclic, and $L_\ast\subset C_\ast$ is a sub-chain complex of $C_\ast$, then any chain map $\hat{F}_\ast: L_\ast \to D_\ast$ can be extended to a chain map $F_\ast:C_\ast \to D_\ast$.  Furthermore, this extension is unique up to chain homotopy.
\end{theorem}
Proofs can be found in \cite{eilenbergSteenrod1952, MosherTangora, MunkresAT}.  In this section, we will extend \cref{thm:acyclic_carrier} to the filtered setting.


\begin{definition}\label{def:alpha_acyclic_complex}
We say a filtered chain complex $C_\ast^T$ is $\alpha$-acyclic if every cycle in $C_\ast^t$ has a boundary in $C_\ast^{\alpha(t)}$.
\end{definition}
This implies that any bar in the persistent homology $H_q(C_\ast^T)$ that is born at $t\in T$ must die before parameter $\alpha(t)$.

\begin{definition}\label{def:alpha_beta_acyclic_carrier}
Let $C_\ast^S, D_\ast^T$ be filtered chain complexes, $\scrC^T:C_\ast^S \to D_\ast^T$ be a filtered carrier, and $\alpha:S\to T$, $\beta:T\to T$ be non-decreasing maps.  We say $\scrC^T$ is $(\alpha, \beta)$-acyclic if $\scrC^T(x)$ is $\beta$-acyclic after $t = \alpha(s)$ for all $x\in C_\ast^s$ and for all $s\in S$.  In the case where $\beta = \id$, then we just say $\scrC^T$ is $\alpha$-acyclic.
\end{definition}
A related definition for cell complexes is to say a carrier $\scrC^T:\X^S \to \Y^T$ is $\alpha$-contractible if $\scrC^T(x)$ is contractible at $t = \alpha(s)$.  This is sufficient to give an $\alpha$-acyclic carrier after application of the chain functor.

\begin{theorem}\label{thm:filtered_acyclic_carrier}
(Filtered acyclic carrier theorem) Let $\scrC^T:C_\ast^S \to D_\ast^T$ be an $(\alpha,\beta)$-acyclic carrier of filtered chain complexes, with $S$ a strict total order with an initial object $0 \in S$.  Let $L_\ast^S \subseteq C_\ast^S$ be a filtered sub-complex generated by a filtered sub-basis of $C_\ast^S$, and $\tilde{F}^\alpha:L_\ast^S \to D_\ast^T$ be an $\alpha$-filtered chain map carried by $\scrC^T$.  Then $\tilde{F}^\alpha$ extends to a filtered chain map $F^{\beta^k \circ \alpha}:C_\ast^S \to D_\ast^T$, where $k$ is the maximal dimension of the chain map, and the extension is unique up to $\beta$-chain homotopy.
\end{theorem}
\begin{proof}
We will proceed by induction on the dimension $k$ of the map, and on the total order on $S$.  First, we start with $\tilde{F}^{0,\alpha(0)}_0:L_0^0\to D_0^{\alpha(0)}$. From the acyclic carrier theorem, \cref{thm:acyclic_carrier}, we can extend to a chain map $F_0^{0,\alpha(0)}\to C_0^0 \to D_0^{\alpha(0)}$.  

Now, let $s > 0$.  Assume that we have extended $F^\alpha_0$ for all $r< s$ so that if $r' < r$,
\begin{equation}\label{eq:extension_restriction_condition}
F^{r,\alpha(r)}_0\mid_{C_\ast^{r'}} = F^{r',\alpha(r')}_0
\end{equation}
Note that this is satisfied trivially for $s=0$.
Let $L_0^{\prime S} = L^S_0 \cup \bigcup_{r < s} C_0^r$, and $\tilde{F}^\alpha_0$ denote the extended map up to all $r < s$.  We can now apply \cref{thm:acyclic_carrier} again to extend to $F^{s,\alpha(s)}$ to $C_0^s$.  Because $S$ is a strict total order, \cref{eq:extension_restriction_condition} continues to be satisfied because the function is extended on each basis element exactly once.  By induction, we can extend to a map of 0-chains $F^\alpha:C_0^S \to D_0^T$.

Because the extension is not necessarily unique, suppose that $F_0^\alpha$ and $G_0^\alpha$ are both extensions of $\tilde{F}^\alpha_0$ carried by $\scrC$.  $\partial_0 (F^\alpha_0 - G^\alpha_0) = 0$, so can be expressed as the boundary of $K_0^{\beta \circ \alpha}:C^S_0\to D^T_1$ after shifting by an additional factor of $\beta$.  This gives a $\beta$ homotopy of 0-chain maps.

Now, we'll extend to higher-dimensional chains for $s=0$.  Assume that we have extended to $F_k^{\beta^k \circ \alpha}:C_k^S\to D_k^T$.  Again, we'll start with the initial object $0$ of $S$. We take $L^{\prime 0}_{\ast\le k+1} = C_{\ast \le k}^0 \cup L_{\ast \le k+1}^0$.  We have extended $F_{\ast \le k}^{\beta^k \circ \alpha}:C_{\ast \le k}^0\to D_{\ast\le k}^{\beta^k \circ \alpha(s)}$.  Let $x\in B_{k+1}$ be a basis element that we must extend at filtration parameter $s = 0$.  We need $\partial_{k+1} F_{k+1} x = F_{k} \partial_{k+1} x$.  The image of the boundary $F_k \partial_{k+1} x$ lies in $D^{\beta^k\circ \alpha (0)}_k$, but since $\scrC$ is $(\alpha,\beta)$-acyclic, the cycle need not have a boundary until we increase the filtration parameter $T$ by another factor of $\beta$.  We can increase the grade on the map $F^{\beta^{k+1}\circ \alpha}$, taking $F^{\beta^{k+1}\circ \alpha} x = \iota^\beta F^{\beta^k\circ \alpha}x$ for $x\in L^{\prime 0}$, and then apply \cref{thm:acyclic_carrier} to extend the map for $x\in C_{k+1}^0$.

Now, we'll extend to higher dimensional chains for $s > 0$. Assume that so far we have satisfied for $r' < r < s$
\begin{equation}\label{eq:equation_restriction_condition_k}
F_{k+1}^{r,\beta^{k+1} \circ \alpha(r)} \mid_{C_k^{r'}} =  F_{k+1}^{r',\beta^{k+1} \circ \alpha(r')}
\end{equation}
and furthermore, that we have shifted the chain maps in lower dimensions via $F^{\beta^{k+1} \circ \alpha} = \iota^\beta F^{\beta^{k} \circ \alpha}$.  Let $x\in B_{k+1}$ via a basis element that we must extend at filtration parameter $s$.  The image of the boundary $F_k\partial_{k+1} x$ lies in $D_k^{\beta^k \circ \alpha(s)}$, and we have already shifted the grade to $\beta^{k+1} \circ \alpha(s)$ at which point the cycle is a boundary of some $y\in D_{k+1}^{\beta^{k+1}\circ \alpha(s)}$ in $\scrC(x)$. Thus, we can extend the map via $F_{k+1}^{\beta^{k+1}\circ \alpha} x = y$.  Again, because $S$ is a strict total order, the map is extended for every basis element exactly once, so \cref{eq:equation_restriction_condition_k} is satisfied.

Following a similar inductive argument, we can extend a $\beta$ homotopy of extended chain maps $F^{\beta^k \circ \alpha}_k$, $G^{\beta^k \circ \alpha}_k$ to a $\beta$ homotopy of $F^{\beta^{k+1} \circ \alpha}_{k+1}$ and $G^{\beta^{k+1} \circ \alpha}_k$, still incurring an additional shift of $\beta$.

By induction on $k$ and the strict total order of $S$, we conclude that we can extend $\tilde{F}^\alpha$ to a shifted chain map $F^{\beta^k \circ \alpha}:C^S_\ast \to D^T_\ast$, and that this chain map is unique up to $\beta$-chain homotopy.
\end{proof}

\begin{remark}
To compute induced maps in homology in dimension $k$, it is only necessary to extend maps up to dimension $k$.  In many cases, $\beta$ will be the identity $\id$, in which case there is no additional penalty for extending to higher-dimensional chains.
\end{remark}

\begin{remark}
In \cref{thm:filtered_acyclic_carrier}
we used the strict total ordering on $S$ to extend the initial map so that we guaranteed that \cref{eq:extension_restriction_condition} is always satisfied.  If $S$ is not a strict total ordering, then additional restrictions on the extension are needed to satisfy this condition.
\end{remark}

\begin{proposition}\label{prop:filtered_aug_preserving_exists}
Let $\scrC:C^S_\ast \to D^T_\ast$ be an $(\alpha,\beta)$-acyclic carrier that is proper with respect to a $T$-filtered basis $B^D_\ast$ of $D_\ast$.  Then there exists a chain map $F^\alpha_0:C_0^S \to D_0^T$ carried by $\scrC$ which preserves the canonical augmentation $\epsilon: x\mapsto 1$ for basis elements $x\in C^S_0$.
\end{proposition}
\begin{proof}
For each 0-dimensional basis element $x\in C_0^S$, we simply assign $F^\alpha_0(x) = y$ for some basis element $y\in B_0^D \mid_{\scrC(x)}$.  Such a $y$ exists at level $\alpha(s)$ for basis elements $x$ at parameter $s$ in $C_0^s$, so the map requires an $\alpha$ shift.  This map will preserve the augmentation of the chain complexes because it sends 0-dimensional basis elements to 0-dimensional basis elements.
\end{proof}
Note that the map $F^\alpha_0$ in \cref{prop:filtered_aug_preserving_exists} can then be extended to $F^{\beta^k\circ \alpha}_{k}$ using \cref{thm:filtered_acyclic_carrier}.

\begin{proposition}\label{prop:pointwise_htpy}
Suppose $F_\ast^\alpha, G_\ast^\alpha: C_\ast^S \to D_\ast^T$ are augmentation-preserving chain maps carried by an $(\alpha, \beta)$-acyclic carrier $\scrC$.  Then $F_\ast$ and $G_\ast$ are $\beta$-chain-homotopic.
\end{proposition}
\begin{proof}
For each basis element $x\in C_0$, $F_0(x), G_0(x)\in \scrC(x)$, and because $F_\ast$ and $G_\ast$ are augmentation preserving, $\epsilon x = \epsilon F(x) = \epsilon G(x)$, so $\epsilon \big(F(x) - G(x)) = 0$.  Because $\scrC$ is $\beta$-acyclic, $\ker \epsilon = \img \partial_1$, so there must exist a 1-chain $K(x)\in \scrC(x)$ at level $\beta \circ \alpha$ so that $\partial_1 K(x) = F(x) - G(x)$, which is a homotopy of zero-chains. We can then apply \cref{thm:filtered_acyclic_carrier} theorem to extend this to a $\beta$-homotopy $K_\ast: F_\ast \to G_\ast$.
\end{proof}
In the case where $\beta =\id$, then the two maps produce isomorphic maps on homology.

\subsection{Interleavings via Filtered Acyclic Carriers}\label{sec:interleavings_via_filtered_carriers}

We'll now turn to examining the conditions under which interleavings can be constructed from filtered carriers. 


\begin{proposition}\label{prop:carrier_interleaving}
Let $\X^S$ and $\Y^T$ be filtered cell complexes, and suppose that $\scrC:\X^S \to \Y^T$ is an $\alpha$-acyclic carrier, $\scrD:\Y^T\to \X^S$ is a $\beta$-acyclic carrier, $\scrA \supseteq \scrD\circ\scrC$ is a $(\beta \circ \alpha)$-acyclic carrier that carries the inclusion map on $\Y^T$, and $\scrB \supseteq \scrC\circ \scrD$ is $(\alpha \circ \beta)$-acyclic and carries the inclusion map on $\X^S$.  Then $H_q(\X^S)$ and $H_q(\Y^T)$ are $(\alpha,\beta)$-interleaved for any $q=0,1,\dots$.
\end{proposition}
\begin{proof}
First, we construct augmentation-preserving shift maps $F^\alpha:C_\ast(\X^s) \to C_\ast(\Y^{\alpha(s)})$ and $G^\beta: C_\ast(\Y^t) \to C_\ast(X^{\beta(t)})$  using \cref{prop:filtered_aug_preserving_exists} and \cref{thm:filtered_acyclic_carrier}.  Now, note that $G^\beta \circ F^\alpha$ is augmentation preserving, and is carried by $\scrD \circ \scrC \subseteq \scrA$ which also carries the inclusion map, so by \cref{prop:pointwise_htpy} $G^\beta \circ F^\alpha \simeq I^\X$.  Similarly, $F^\alpha \circ G^\beta \simeq I^\Y$.  Thus, the maps $F^\alpha$ and $G^\beta$ give an $(\alpha,\beta)$-interleaving on homology.
\end{proof}

In practice, more specific situations reduce the number of conditions that we need to satisfy. Often, we will find it convenient to take $\scrA = \scrC \circ \scrD$, and $\scrB = \scrD \circ \scrC$ when we can show that the composites are acyclic and carry inclusions.

\begin{corollary}\label{cor:surjection_interleaving}
Suppose $f^\alpha: \X^s \to \Y^{\alpha(s)}$ is a surjective simplicial map for every $s\in S$, and suppose $\scrC:\Y^T \to \X^S$, defined by $\scrC(y) = \langle f^{-1}(y) \rangle$ be a $\beta$-acyclic carrier. Then $H_q(\X^S)$ and $H_q(\Y^T)$ are $(\alpha,\beta)$-interleaved for $q=0,1,\dots$. 
\end{corollary}
\begin{proof}
Because $f^{\alpha}$ is simplicial, the carrier $\scrC^f$ defined by $\scrC^f(x) = \langle f(x) \rangle$ is an $\alpha$-acyclic carrier that carries $f^\alpha$.
Because $f^{\alpha}$ is a surjective simplicial map, $\scrC(y)$ is nonempty and maps to proper sub-complexes of $\X^S$ for each $y\in \Y^T$, so is a well-defined filtered carrier.  By definition, of $\scrC$, the composition $\scrC \circ \scrC^f$ carries the inclusion map $\iota^\X$.  Additionally, $\scrC\circ f^\alpha$ is $(\beta\circ\alpha)$-acyclic, because $\scrC$ is $\beta$-acyclic for the simplex $f^\alpha(x)$ for each $x\in \X^S$.
Because $\scrC(y) = \langle f^{-1}(y) \rangle$, $f^\alpha\circ \scrC(y) = \langle y \rangle$, which is a simplicial carrier and thus acyclic. Note that $y \in f^\alpha\circ \scrC(y)$, so $f^\alpha \circ \scrC$ carries $\iota^\Y$. We can now apply the chain functor and \cref{prop:carrier_interleaving} to complete the proof.
\end{proof}




\section{Cover Complexes}\label{sec:covers}


\subsection{Local Stability}\label{sec:cover_local_interleavings}

In classical topology, a situation of interest is to study spaces over a base space.  In particular, we consider surjective maps $p: \calX \to \calB$, where $\calB$ is called the base space.  Some problems of interest focus on maps over $\calB$.
\begin{equation}
\begin{tikzcd}
\calX\ar[rr,"f"]\ar[dr,"p"] && \calY\ar[dl,"q"]\\
& \calB&
\end{tikzcd}
\end{equation}

\begin{definition}\label{def:compatible_system_carriers}
Let $\X^S(\calU)$, $\Y^T(\calU)$ be cover complexes over a cover $\calU$.  A system of carriers $\scrC(\calU):\X^S(\calU) \to \Y^T(\calU)$ consists of carriers $\scrC(U): \X^S(U) \to \Y^T(U)$ for each $U\in \calU$. We say the system of carriers is {\em compatible} if $\cap U_k \ne \emptyset$ implies $\scrC(U_i)\mid_{\cap U_k} = \scrC(U_j)\mid_{\cap U_k}$ for all $U_i, U_j \in \{U_k\}$. 
\end{definition}

In general, $U\in \calU$ need not cover the same points in $X$ and $Y$ (denoting the vertex sets of $\X$, $\Y$ respectively).  We can alternatively think of it as an identification of sets in covers of each vertex set, or a set in a cover of the disjoint union $X \sqcup Y$.

When the system of carriers is compatible, we can extend carriers defined on sets of $\calU$ to intersections via $\scrC(\cap U_k) = \scrC(U_i)\mid_{\cap U_k}$ for $U_i\in \{U_k\}$.  We'll say a compatible system of carriers is $(\alpha,\beta)$-acyclic if $\scrC(\cap U_k)$ is $(\alpha,\beta)$-acyclic for all $\{U_k\} \subset \calU$ where $\cap U_k \ne \emptyset$.

We can define a carrier $\scrC_\calU:\X^T(\calU) \to \Y^S(\calU)$ from a compatible system of carriers via $\scrC_\calU(x) = \scrC(\cap \{ U \ni x\})(x)$.  When a compatible system of carriers $\scrC(\calU)$ is $(\alpha,\beta)$-acyclic, $\scrC_\calU$ is also $(\alpha,\beta)$-acyclic through application of the definition.  The advantage of using a compatible system of carriers $\scrC(\calU)$ instead of the global carrier $\scrC_\calU$ is that we only need to check conditions locally in the cover.

\begin{proposition}\label{prop:interleaved_cover_cpxs}
Let $\calU$ be a finite cover.  Suppose $\scrC(\calU):\X^S(\calU) \to \Y^T(\calU)$ is an $\alpha$-acyclic compatible system of carriers, and $\scrD(\calU):\Y^T(\calU) \to \X^S(\calU)$ is a $\beta$-acyclic compatible system of carriers.  Furthermore suppose that for each $V = \cap U_k \ne \emptyset$, that $\scrC(V) \circ \scrD(V)$ is $(\alpha\circ\beta)$-acyclic and carries the identity, and $\scrD(V) \circ \scrC(V)$ is $(\beta\circ\alpha)$-acyclic and carries the identity. Then there exists an $(\alpha,\beta)$-interleaving of $H_q(\X^S)$ and $H_q(\Y^T)$. 
\end{proposition}
\begin{proof}
This follows by constructing the global carriers $\scrC_\calU:\X^S(\calU) \to \Y^T(\calU)$ and $\scrD_\calU:\Y^T(\calU)\to \X^S(\calU)$, and noting that because the composite $\scrC_\calU \circ \scrD_\calU$ is $(\alpha\circ\beta)$-acyclic locally and carries the identity locally, it satisfies these properties globally.  Similarly, $\scrD_\calU \circ \scrC_\calU$ is $(\beta\circ\alpha)$-acyclic and carries the identity.  We can then apply \cref{prop:carrier_interleaving} to obtain the result.
\end{proof}


The utility of \cref{prop:interleaved_cover_cpxs} is to show that if we have identified sub-complexes of $\calX^S$ and $\calY^T$ in a consistent way using the cover that we can interleave the homology of the two filtrations.

\subsection{Local Geometric Stability}
\Cref{prop:interleaved_cover_cpxs} can be used to extend standard geometric stability results as in \cite{geometric_stab2014}
to cover complexes. We will focus on how cover complexes behave with respect to perturbations of the data

Several of our results will use the refinement of the cover $\calU$, as
\begin{equation}
\calUb = \big\{\bigcap_{i\in I} U_i | \{U_i\}_{i\in I} \subseteq \calU\big\}
\end{equation}

\begin{proposition}\label{prop:nerve_equivalence}
$H_\ast(\calN(\calU)) \simeq H_\ast(\calN(\calUb))$.
\end{proposition}
\begin{proof}
We define a carrier $\scrC:\calN(\calU)\to \calN(\calUb)$ via
\begin{equation}
    \scrC(U_0,\dots,U_k) = \langle V = \cap \{U_i\} \mid \{U_i\} \subseteq \{U_0,\dots,U_k\} \rangle
\end{equation}
This carrier is acyclic because it forms a cone with the vertex $V = \cap_{i=0}^k U_i$.

We define a carrier $\scrD:\calN(\calUb)\to \calN(\calU)$ via
\begin{equation}
    \scrC(V_0,\dots,V_k) = \langle U \mid U \supset V_i \in \{V_0,\dots,V_k\}\rangle
\end{equation}
Note that if $(V_0,\dots,V_k)$ is a simplex in $\calN(\calUb)$, there is a smallest set $V_{i_0}$ in the simplex, and so $\scrC(V_0,\dots,V_k) = \scrC(V_{i_0})$.  This also implies that $\scrC$ is simplicial, thus acyclic.

Now, we have that $\scrD \circ \scrC(U_0,\dots,U_k)$ is a the simplex $(U_0,\dots,U_k,U'_0,\dots)$ where the extra simplices $U'_i$ are added if $\cap_{i=0}^k U_i \subseteq U
'_i$, which can appear for degenerate $\calUb$.  This carrier is simplicial, thus acyclic, and clearly carries the identity.  

The composition $\scrC\circ \scrD$ is acyclic because $\scrC\circ \scrD(V_0,\dots,V_k)$ forms a cone with the vertex on the minimal element $V_{i_0}\in \{V_0,\dots,V_k\}$.  This composite carrier also carries the identity map.

We can now apply \cref{prop:carrier_interleaving} to trivial filtrations on the Nerves to obtain the result.
\end{proof}

\begin{proposition}\label{cor:rips_cover_hausdorff}
Let $\bX, \bY$ be samples from a metric space $(X, d_X)$, and let $\calU$ be a cover of $\bX$ and $\calV$ be a cover of $\bY$. Suppose that for all $U\in \calUb$, there exists a $V\in \calVb$ such that $d_H(U, V)\le \epsilon$.
Then $\calR(\bX, \calU; r)$ and $\calR(\bY, \calV; r)$ are $2\epsilon$-interleaved. 
\end{proposition}

\begin{proof}
Let $U_x = \bigcap \{U\in \calU \mid U \ni x\}$.  By assumption, there exists some $V_x\in \calVb$ such that $d_H(U_x, V_x) \le \epsilon$, meaning there must exist some $y\in V_x$ so that $d_X(x,y)\le \epsilon$.  Let $\Omega\subseteq \bX\times \bY$ be the left-total relation $\Omega(x) = \{ y \in V_x \mid d_X(x,y) \le \epsilon\}$.  Then the induced carrier $\scrC_\Omega:\calR(\bX, \calU; r) \to \calR(\bY, \calV; r)$ is $2\epsilon$-simplicial.  Similarly, for $y\in \bY$, we take $V_y = \bigcap_{V\ni y} V$ and $U_y\in \calUb$ the set that satisfies $d_H(U_y, V_y) \le \epsilon$.  From the set using the right-total relation $\Psi\subseteq \bX\times \bY$, with $\Psi(y) = \{x \in U_y \mid d_X(x,y) \le \epsilon\}$, we obtain a $2\epsilon$-simplical carrier $\scrD_\Psi:\calR(\bY, \calV; r) \to \calR(\bX, \calU; r)$.

Now, note that the composite carrier $\scrD_\Psi \circ \scrC_\Omega$ need not carry the identity, because $y\in V_x$ does not imply $x\in U_y$.  However, $y\in V_x$ implies does imply that $V_y \subseteq V_x$ which combined with the Hausdorff distance bound implies there must exist some $x'\in V_y \cap \bX$ such that $d_X(x',y)\le \epsilon$, which implies $d_X(x',x) \le 2\epsilon$ by triangle inequality.  We can define a left-total relation $\Omega' \subseteq \bX\times \bX$, with $\Omega' = \{(x,x') \mid d(x,x') \le 2\epsilon, x'\in V_x\}$, which is nonempty, and $4\epsilon$-simplical by triangle inequality.  Furthermore, the carrier $\scrA_{\Omega'}$ contains the composite $\scrD_\Psi \circ \scrC_\Omega$ and carries the identity.  Similarly, we can define a relation $\Psi'\subset \bY \times \bY$ with $\Psi' = \{(y,y') \mid d_X(y,y') \le 2\epsilon\}$ which produces a $4\epsilon$-simplicial carrier $\scrB_{\Psi'}$ which contains the composite $\scrC_\Omega \circ \scrD_\Psi$ and carries the identity.

We can now apply \cref{prop:carrier_interleaving} to obtain the result.
\end{proof}

\begin{corollary}\label{cor:rips_cover_hausdorff_joint}
Let $\bX, \bY$ be samples from a metric space $(X, d_X)$, and let $\calW$ be a cover of $\bX \sqcup \bY$ such that $d_H(W|_\bX, W|_\bY) \le \epsilon$ for all $W \in \calWb$.  Then $\calR(\bX, \calW; r)$ and $\calR(\bY, \calW; r)$ are $2\epsilon$-interleaved. 
\end{corollary}

\begin{proof}
We apply \cref{cor:rips_cover_hausdorff} taking $\calU = \{ W \cap \bX \mid W\in \calW\}$ and $\calV = \{W \cap \bY \mid W\in \calW\}$.
\end{proof}

\cref{cor:rips_cover_hausdorff_joint} specializes to the standard stability bound \cite{geometric_stab2014} when $\calU = \{\bX\}$.

\begin{comment}
The conditions of \cref{cor:rips_cover_hausdorff} imply that $\calN(\calU) \simeq \calN(\calV)$.
\end{comment}

\subsection{A Generalized Nerve Theorem}\label{sec:gen_nerve_theorem}

We'll now prove a version of the Nerve theorem for cover complexes.  This result can be viewed as a special case of of the approximate nerve theorems in \cite{ApproximateNerveTheorem2017, cavannaGeneralizedPersistentNerve2018}.  While our proof is narrower in scope than the aforementioned results, the use of carriers will considerably simplify the proof, compared to \cite{ApproximateNerveTheorem2017} which used the Mayer-Vietoris spectral sequence, and \cite{cavannaGeneralizedPersistentNerve2018} which used a construction using the blowup complex.

\begin{theorem}\label{thm:nerve_theorem}
(Nerve Theorem \cite{borsukNerve}) Let $\calU$ be a cover of a paracompact space $X$, where if $\cap U_i\ne \emptyset$, then $\cap U_i$ is contractible.  Then $\calN(\calU)\simeq X$.
\end{theorem}
A proof can be found in \cite{HatcherAT}. 

\begin{theorem}\label{thm:acyclic_nerve_theorem}
[an $\alpha$-Acyclic Nerve Theorem] Let $\calU$ be a cover of a vertex set $X$, and let $\X^T(\calU)$ be a simplicial cover complex, with $T$ a strict order with initial object $0$.  If $\X^T(V)$ is $\alpha$-acyclic for every $V\in \calUb$, then $H_k(\calN(\calU))$ and $H_k(\X^T(\calU))$  are $(\alpha^{k+1}, \id)$-interleaved.
\end{theorem}
\begin{proof}
We'll construct an interleaving with $\calN(\calUb)$, which has isomorphic homology to $\calN(\calU)$ by \cref{prop:nerve_equivalence}.

We'll first define a carrier $\scrD:\calN(\calUb)\to \X^T$.  We take $\scrD(V) = \X^T(V)$, and $\scrD(V_0,\dots,V_k) = \X^T(V_0\cup\dots\cup V_k) = \X^T(V_{i_k})$, where $V_{i_k}$ is the maximal set in $\{V_0,\dots,V_k\}$.  This forms a $(0,\alpha)$-acylic carrier by assumption, where $0$ denotes the map to the initial object of $T$.

Now, we define a carrier $\scrC:\X^T(\calU)\to \calN(\calUb)$.  We take
\begin{equation}\label{eq:nerve_carrier_C}
\scrC(x_0,\dots,x_k) = \bigg\langle \bigg\{\bigcap_{V \supseteq S} V \bigg\}_{S\subseteq \{x_0,\dots,x_k\}}\bigg\rangle
\end{equation}
Let $V' = \bigcap_{V\supseteq \{x_0,\dots,x_k\}} V$.
The carrier above forms a cone with $V'$, so is acyclic.

$\scrD\circ \scrC$ carries the identity because \cref{eq:nerve_carrier_C} ensures that some $V'$ for which $\{x_0,\dots,x_k\}\subseteq V'$ is included in $\scrC(x_0,\dots,x_k)$, and $\scrD(V) \ni (x_0,\dots,x_k)$ for that $V$.  Because all other sets in \cref{eq:nerve_carrier_C} are contained in $V'$, $\scrD\circ \scrC(x_0,\dots,x_k) = \scrD(V')$, which is $(0,\alpha)$-acyclic by assumption.

Any $(x_0,\dots,x_k)\in \scrD(V)$, implies $\{x_0,\dots,x_k\}\subseteq V$.  Thus, every $V_i$ generating the carrier in \cref{eq:nerve_carrier_C} satisfies $V_i \subseteq \scrC(x_0,\dots,x_k)$.  We can define $\scrA(V)$ to be the star of $V$ inside $\calN(\calUb)$.  This carrier is acyclic because it forms a cone with the vertex for $V$, and contains $\scrC\circ \scrD(V)$.  For $(V_0,\dots,V_k)\in \calN(\calUb)$, we take $\scrA(V_0,\dots,V_k) = \scrA(V_{i_k})$, where $V_{i_k}$ is the maximal set in the simplex.  Again, this carrier is acyclic and carries the identity. 

We have now constructed carriers for maps in the following diagram
\begin{equation}
\begin{tikzcd}
\X^T(\calU) \ar[r, hookrightarrow]\ar[d,"\scrC"] &\X^T(\calU)\ar[d,"\scrC"]\\
\calN(\calUb) \ar[r,"\scrA"] \ar[ru,"\scrD"] &\calN(\calUb)
\end{tikzcd}
\end{equation}
We can now construct a map $P_\ast :C_\ast(\X^T(\calU))\to C_\ast(\calN(\calUb))$ carried by $\scrC$ by applying \cref{prop:filtered_aug_preserving_exists}. 
We can also construct maps $F_i^{\alpha^i}: C_i(\calN(\calUb))\to C_i(\X^{\alpha^{i}(0)})$ using \cref{thm:filtered_acyclic_carrier}, where $\partial_i F_i^{\alpha^i} x = F_{i-1}^{\alpha^{i-1}} \partial_i x$, which we need to construct for $i=0,\dots,k+1$.  Because $\scrD\circ \scrC$ carries the inclusion, we can construct a homotopy, but only after increasing the grade by an extra factor of $\alpha$ in each dimension $i$, $I^{\alpha} F_i^{\alpha^i}\circ P_i \simeq I^{\alpha^{i+1}}_i$.  In order to compute induced maps on homology for $H_i$, we only need to extend the chain homotopy up to dimension $i$.  On homology, we have $\tilde{I}^\alpha \tilde{F}^{\alpha^{k}}_k {\tilde{P}_k} \cong \tilde{I}^{\alpha^{k+1}}_k$.

Finally, because $\scrA$ is acyclic and carries $P_\ast \circ F_\ast$ as well as the inclusion, we have $\tilde{P}_k \circ I^{\alpha} \tilde{F}^{\alpha^k} \simeq I$, we have constructed a $(\alpha^{k+1},\id)$-interleaving.
\end{proof}

Note that for Vietoris-Rips cover complexes as well as other geometric complexes, that there will be some parameter $t\in T$ at which $\X^T(V)$ will be acyclic for all $V$, when $\X^T(V)$ forms the maximal simplex on its vertex set.  At this point, the cover complex and nerve are homotopic by the standard nerve theorem (\cref{thm:nerve_theorem}).

\begin{corollary}\label{cor:acyclic_nerve_cor}
Let $\calU$ be a cover of $X$, where $\X^T(\calU)$ satisfies the conditions of \cref{thm:acyclic_nerve_theorem}.  Then if $\calN(\calU)$ is acyclic, $H_k(\X^T(\calU))$ is $(\alpha^{k+1})$-acyclic.
\end{corollary}
\begin{proof}
This follows because if $\calN(\calU)$ is acyclic, then the interleaving implies that $\X^T(\calU)$ is $\alpha^{k+1}$-acyclic.
\end{proof}

\section{Rips-Cover Constructions}\label{sec:application_to_rips}

We now focus on Vietoris-Rips cover complexes, which we denote as $\calR(\bX, \calU; r)$. We seek to answer the following questions:
\begin{enumerate}
    \item For a fixed cover $\calU$, how sensitive is $\calR(\bX, \calU; r)$ to perturbations of the underlying data $X$?
    \item For a fixed dataset $\bX$, how sensitive is $\calR(\bX, \calU; r)$ to the choice of cover $\calU$?
    \item How does $\calR(\bX, \calU; r)$ relate to the full Vietoris-Rips complex $\calR(\bX;r)$?
\end{enumerate}

A related definition is the {\em Rips system} found in Yoon's 2018 dissertation \cite{yoon2018} which is used for distributed computation of persistent homology of Rips complexes via cellular (co)-sheaves.  Yoon shows that if the Nerve is 1-dimensional, and the system covers the full Rips complex, that the Rips system can be used to obtain the Homology of the full complex, and develops a distributed algorithm for computation.  We will consider more general coverings, and characterize regimes where the cover complex and full complex are interleaved, but not identical.  Distribution schemes for computing persistent homology of cover complexes in their full generality are beyond the scope of this work.

\subsection{Interleavings for Arbitrary Covers}\label{sec:cover_full_interleaving}

We now turn to relating the persistent homology of $\calR(\bX, \calU; r)$ to the persistent homology of $\calR(\bX; r)$.  At large $r$ parameters, Vietoris-Rips complexes become acyclic, so following \cref{thm:acyclic_nerve_theorem} that $PH_\ast(\calR(\bX, \calU; r))$ will eventually converge to $H_\ast(\calN(\calU))$.  This means that unless $\calN(\calU)$ is acyclic, $PH_\ast(\calR(\bX, \calU; r)$ and $PH_\ast(\calR(\bX; r))$ can not possibly interleave for sufficiently large $r$ parameters.  However, in situations where sets in the cover have non-trivial structure, we would like to understand how this structure relates to the full filtration $\calR(\bX; r)$, particularly for small values of $r$.  

Because there are inclusions $\calR(\bX, \calU; r) \hookrightarrow \calR(\bX; r)$, it suffices to study under what conditions we can extend a map $f^\alpha$ in the diagram
\begin{equation}\label{eq:cover_interleaving}
\begin{tikzcd}
\calR(\bX, \calU; r) \ar[r, hookrightarrow] \ar[d, hookrightarrow]& \calR(\bX, \calU; \alpha(r))\ar[d, hookrightarrow]\\
\calR(\bX; r) \ar[r, hookrightarrow]\ar[ur, "f^\alpha"] & \calR(\bX; \alpha(r))
\end{tikzcd}
\end{equation}

We focus on a carrier  $\scrC: \calR(\bX; r) \to \calR(\bX, \calU; r)$ generated from witness sets
\begin{equation}
\bX(x_0,\dots,x_k) = \{ y\in \bX \mid d(y,x_i) \le d(x_0,\dots,x_k)\ \forall i=0,\dots,k\}
\end{equation}
and their union, denoted
\begin{equation}
\bar{\bX}(x_0,\dots,x_k) =\bigcup_{S \in \calP(\{x_0,\dots,x_k\})}\bX(S)
\end{equation}
where $\calP$ denotes the power set.  We define the carrier $\scrC: \calR(\bX; r) \to \calR(\bX, \calU; r)$ via
\begin{equation}\label{eq:rips_cover_carrier}
    \scrC: (x_0,\dots,x_k) \mapsto \langle\bar{\bX}(x_0,\dots,x_k) \rangle
\end{equation}
and let
\begin{equation}
\calUb(x_0,\dots,x_k) = \{V \cap \bar{\bX}(x_0,\dots,x_k)  \mid  V\in \calUb, \bar{\bX}(x_0,\dots,x_k) \cap V \ne \emptyset \}
\end{equation}
which covers $\scrC(x_0,\dots,x_k)$.

\begin{definition}\label{def:rips_inter_thres}
We define three thresholds: $R_1\le R_2 \le R_3$ which describe different regimes of the non-decreasing map $\alpha$.
\begin{enumerate}
    \item Let $R_1$ be the largest value so that if $d(x_0,\dots,x_k)\le R_1$ then there exists some $U\in \calU$ so that $x_0,\dots,x_k\in U$.
    \item Let $R_2$ be the largest value so that if $d(x_0,\dots,x_k) \le R_2$ then $\calN(\calUb(x_0,\dots,x_k))$ is acyclic and $\bX(x_0,\dots,x_k)\cap V)$ is non-empty for each $V\in \calUb(x_0,\dots,x_k)$.
    \item Let $R_3$ be the largest value so that if $d(x_0,\dots,x_k) \le R_3$ then $\calN(\calUb(x_0,\dots,x_k))$ is acyclic.
\end{enumerate}
\end{definition}

If we impose a mild condition that for any points $x,y\in \bX$ with $d(x,y) = 0$ then $U\cap \{x, y\}$ is either $\{x,y\}$ or empty for all $U\in \calU$, then $0\le R_1$.

\begin{theorem}\label{thm:rips_cover_interleaving}
$H_k(\calR(\bX, \calU; r))$ and $H_k(\calR(\bX; r))$ are $(\id,\alpha)$-interleaved, where $\alpha(r) = r$ for $r \le R_1$, $\alpha(r) \le 2r$ for $r\le R_2$ and $\alpha(r) \le 3r$ for $r\le R_3$.  For $r > R_3$, an interleaving may not exist.
\end{theorem}
Proofs of each inequality are in \cref{prop:cover_rips_identical}, \cref{prop:cover_rips_inter_2r}, and \cref{prop:cover_rips_inter_3r}.

\begin{proposition}\label{prop:cover_rips_identical}
$\calR(\bX, \calU; r) = \calR(\bX; r)$ for all $r\le R_1$.
\end{proposition}
\begin{proof}
This follows because if $(x_0,\dots,x_k)\in \calR(\bX; r)$, then $d(x_0,\dots,x_k)\le r \le R_1$, so $(x_0,\dots,x_k)\in \calR(\bX, U; r)\subseteq \calR(\bX, \calU; r)$ for some $U\in \calU$.  Thus $\calR(\bX;r) \subseteq \calR(\bX, \calU; r)$, and we already know $\calR(\bX, \calU; r) \subseteq \calR(\bX; r)$, giving equality.
\end{proof}

This mean that covers $\calU$ that encode some notion of locality produce cover complexes which are identical to the full Rips complex at the beginning of the filtration. 


We now turn to the non-trivial interleavings. Let $\iota:\calR(\bX, \calU; r) \to \calR(\bX; r)$ denote the canonical inclusion, seen in \cref{eq:cover_interleaving}.  Clearly, $\scrC \circ \iota$ carries the inclusion $\calR(\bX, \calU; r) \to \calR(\bX, \calU; \alpha(r))$.
However, the carrier $\iota \circ \scrC$ does not carry the inclusion for any simplices in $\calR(\bX; r)$ that are not in the cover complex $\calR(\bX, \calU; r)$.  We need to find another carrier which does carry the inclusion which also contains this carrier.  Consider $\calD: R(\bX;r) \to \calR(\bX; r)$, defined as
\begin{equation}\label{eq:rips_cover_carrier2}
    \scrD: (x_0,\dots,x_k) \mapsto \langle\bar{\bX}(x_0,\dots,x_k) \rangle.
\end{equation}
The difference between $\scrC$ and $\scrD$, despite the similarity of their definitions is that they map to different complexes. $\scrC$ maps to subcomplexes of $\calR(\bX, \calU; r)$, and $\scrD$ maps to subcomplexes of $\calR(\bX; r)$.  Note that $\scrD$ does carry $\iota \circ \scrC$.

If $\scrD$ is also $\alpha$-acyclic, we can apply \cref{prop:carrier_interleaving} to construct the interleaving.  The remainder of this section describes conditions that will allow us to bound the non-decreasing map $\alpha$.

\begin{lemma}
If $(\bX, d)$ is a metric space, then $\calD$ is acyclic for $\alpha: r\mapsto 2r$.
\end{lemma}
\begin{proof}
Consider $\calD(x_0,\dots,x_k)$, and let $r = d_X(x_0,\dots,x_k)$.  Without loss of generality, consider distances to $x_0$.  Let $y\in \calD(x_0,\dots,x_k)$.  By definition of $\scrD$, either $d(y,x_0)\le r$, or $d(y,x_i) \le r$ for some $x_i \in \{x_1,\dots,x_k\}$.  Because $d_X(x_0,x_i)\le r$, by triangle inequality, $d(y,x_0)\le 2r$.  Because the Vietoris-Rips complex is a flag complex, this implies  $\calD(x_0,\dots,x_k)$ forms a cone with $x_0$ at parameter $2r$ and so is acyclic.
\end{proof}

The more difficult carrier to analyze is $\scrC$.  We'll consider the restriction of the cover to the carrier.    If $\calN(\calUb(x_0,\dots,x_k))$ is acyclic for each $(x_0,\dots,x_k)\in \calR(\bX; r)$, and $\calR(\bar{\bX}(x_0,\dots,x_k), V; r)$ is $\alpha$-contractible, then $\scrC$ is $\alpha$-acyclic by the Nerve theorem.

\begin{lemma}\label{lem:restricted_cover_3r}
Let $r = d_X(x_0,\dots,x_k)$.  For each $V\in \calUb(x_0,\dots,x_k)$, $\calR(V; 3r)$ is contractible.
\end{lemma}
\begin{proof}
Let $y,y'\in V$.  Then there are some $x,x'\in \{x_0,\dots,x_k\}$ for which $d(y,x), d(y',x') \le r$.  Because $d(x,x')\le r$, by triangle inequality $d(y,y')\le 3r$.  Thus, $\calR(V;3r)$ forms a simplex, so is contractible. 
\end{proof}

In general, the bound in \cref{lem:restricted_cover_3r} can be pessimistic.  For instance,
\begin{lemma}\label{lem:restricted_cover_2r}
Let $r \le R_2$, so that for $d_X(x_0,\dots,x_k) \le r$, $\bX(x_0,\dots,x_k)\cap V$ is non-empty for each $V\in \calUb(x_0,\dots,x_k)$.  Then $\calR(V; 2r)$ is contractible.
\end{lemma}
\begin{proof}
Fix $V\in \calUb$.  By assumption, there is some $y\in V$ so that $d(y,x_i) \le r$ for all $i=0,\dots,k$.  For some other $y'\in V$, we have $d(y,x_i) \le r$ for some $i=0,\dots,k$.  By triangle inequality, $d(y,y')\le 2r$.  Since this holds for all $y'\in V$, $\calR(V; 2r)$ forms a cone with $y$, and is thus contractible.
\end{proof}

We can now tie things together in the following propositions.
\begin{proposition}\label{prop:cover_rips_inter_2r}
$H_k(\calR(\bX, \calU; r))$ and $H_k(\calR(\bX; r))$ are $(\id, 2r)$-interleaved for $r \le R_2$.
\end{proposition}
\begin{proposition}\label{prop:cover_rips_inter_3r}
$H_k(\calR(\bX, \calU; r))$ and $H_k(\calR(\bX; r))$ are $(\id, 3r)$-interleaved for $r \le R_3$.
\end{proposition}
\begin{proof}
We use the approximate nerve theorem, \cref{thm:acyclic_nerve_theorem}, to show that $\calC$ is acyclic under the conditions of \cref{prop:cover_rips_inter_2r} and \cref{prop:cover_rips_inter_3r}. 
By definition of $R_3$, the nerve $\calN(\calUb(x_0,\dots,x_k))$ is acyclic in both propositions. \cref{lem:restricted_cover_2r} or \cref{lem:restricted_cover_3r} ensures that each set in the nerve is $2r$ or $3r$ acyclic respectively, so the whole carrier is $2r$ or $3r$ acyclic respectively.
\end{proof}


Note that the sets in the covers $\calU$ do not need to be acyclic at the levels prescribed, but rather their restriction to points within a certain distance of each simplex.  This means there can be a variety of non-trivial structure in each set in the cover.

\subsection{Sparse Filtrations via Covers}\label{sec:sparse_filt_cover}

\Cref{thm:rips_cover_interleaving} can be applied to any cover $\calU$ of a data set $\bX$, and to a certain extent can guide the selection of a cover $\calU$ that increases $R_1, R_2$, and $R_3$ as much as possible:
\begin{enumerate}
\item $R_1$ is determined by the threshold at which for all $x\in \bX$, there exists some set $\calU$ which contains all its $R_1$-nearest neighbors.
\item To maximize $R_2$, we want to ensure that the cover $\calU$ contains witnesses to simplices.  This may require sets covering large distances in sparse regions.
\item To maximize $R_3$, we want to make $\calN(\calUb(x_0,\dots,x_k))$ acyclic for all simplices $(x_0,\dots, x_k)$.  This requires sufficient overlap of sets in cover.
\end{enumerate}

If the goal is to construct a cover that gives an interleaving for all filtration values, a practical approach is to construct a sparse filtration, as originally proposed by Sheehy \cite{sheehyLinearSizeApproximationsVietoris2013}.  We consider a variant of this approach using a Vietoris-Rips cover complex which is amenable to a straightforward analysis. Another, more geometric, approach based on persistent nerves is studied in \cite{cavanna_geometric_2015}.  The key differences between the approach here and \cite{sheehyLinearSizeApproximationsVietoris2013,cavanna_geometric_2015} are that the Vietoris-Rips cover complexes are not generally flag complexes, and that we do not consider re-weighting of edges to tighten the interleaving.

Consider a nested sequence of greedily chosen landmark sets $\bL_0 \subset \bL_1 \subset \dots \subset \bL_n = \bX$, so $\bL_i = \bL_{i-1} \cup \{x_i\}$ where $x_0$ can be chosen arbitrarily, and $x_i\in \bX$ is a point that realizes the Hausdorff distance $d_H(\bL_{i-1}, \bX)$. Let $\lambda_i = d_H(\bL_{i-1}, \bX)$, with $\lambda_0 = \infty$, and let $c > 1$ be a fixed constant.  We construct a cover $\calU$ of the data set $\bX$ by associating a set $U_x$ to each element $x\in \bX$
\begin{equation}
    U_x = \bigcup_{i=1}^n \{ \ell \in \bL_i \mid  d(\ell, x) \le c \lambda_i\}.
\end{equation}
Each set is non empty because $x\in \bL_n$ and $d(x,x) = 0$ implies $x\in U_x$.  Furthermore, the Nerve of the cover $\calN(\calU)$ is acyclic, as the single point $x_0$ contained in $\bL_0$ is contained in every set, every intersection of sets in $\calU$ is non-empty.

\begin{lemma}\label{lem:insertion_lb}
Let $x_0,\dots, x_q\in \bX$, with $d(x_0,\dots,x_q) = r$.  Then for any $\lambda_i \ge \frac{r}{c-1}$, there exists some $\ell \in \bL_i$ so that $\ell \in \cap_{i=0}^k U_{x_i}$.
\end{lemma}
\begin{proof}
We consider a landmark set $\bL\subseteq \bX$ with $d_H(\bL, \bX) = \epsilon$.  Without loss of generality, we take the point $x_0$, for which there must exists some $\ell\in \bL$ such that $d(x_0, \ell) \le \epsilon$ by the Hausdorff distance bound, so $\ell \in U_{x_0}$.  In order to guarantee that $\ell$ is in $U_{x_i}, i=1,\dots,q$, we must satisfy
\begin{align*}
    d(x_i,\ell) \le d(x_i,x_0) + d(x_0,\ell) &\le c \epsilon\\
    r + \epsilon &\le c \epsilon\\
    r &\le (c-1) \epsilon\\
    \epsilon &\ge \frac{r}{c-1}
\end{align*}
Thus, for any $\bL_i$ with $\lambda_i = d_H(\bL_{i-1}, \bX)\ge \frac{r}{c-1}$, there exists such an $\ell$.
\end{proof}
\begin{lemma}\label{lem:witness_ub}
Let $x_0,\dots, x_q\in \bX$, with $d(x_0,\dots,x_q) = r$, and let $i$ be the index such that $\lambda_i \ge \frac{r}{c-1}$ and $\lambda_{i+1} < \frac{r}{c-1}$, then there exist $\ell_j\in \bL_i, j=0,1,\dots,q$ such that $d(x_j, \ell_j) \le \lambda_{i+1} < \frac{r}{c-1}$.
\end{lemma}
\begin{proof}
We have that $\lambda_{i+1} = d_H(\bL_i, \bX) < \frac{r}{c-1}$, so for all $x\in \bX$, there is some $\ell\in \bL_i$ so that $d(x,\ell) \le \lambda_{i+1} < \frac{r}{c-1}$.
\end{proof}
\begin{proposition}\label{prop:sparse_witness}
Let $x_0,\dots, x_q\in \bX$, with $d(x_0,\dots,x_q) = r$.  Then there exists an $\ell \in \cap_{j=0}^q U_{x_j}$ with $d(x_j, \ell) \le \frac{cr}{c-1}$ for all $j=0,\dots,q$.
\end{proposition}
\begin{proof}
Let $i$ be the index such that $\lambda_i \ge \frac{r}{c-1}$ and $\lambda_{i+1} < \frac{r}{c-1}$.  From \cref{lem:witness_ub}, there is an $\ell \in \bL_i$ with $d(\ell, x_0) < \frac{r}{c-1}$, and since $\bL_i \ge \frac{r}{c-1}$, $\ell \in U_{x_0}$.  Now, from \cref{lem:insertion_lb}, such an $\ell$ is also in $U_{x_j}$ for $j=1,\dots,q$.
\end{proof}
This motivates the construction of a carrier $\scrC:\calR(\calX; r) \to \calR(\calX, \calU; r)$.  Let 
\begin{equation}\label{eq:sparse_carrier}
    \scrC(x_0,\dots,x_q) = \bigg\langle  \{\ell(x_0,\dots,x_q)\} \cup \bigcup_{\sigma \in \calP(x_0,\dots,x_k)} \scrC(\sigma) \bigg\rangle
\end{equation}
where $\ell(x_0,\dots,x_q)$ is an arbitrary choice of $\ell$ which satisfies \cref{prop:sparse_witness}.
\begin{proposition}\label{prop:sparse_carrier}
Let $d(x_0,\dots,x_q) = r$. Then the carrier in \cref{eq:cover_interleaving} is acyclic at level $\frac{2c}{c-1} r$.
\end{proposition}
\begin{proof}
We consider when the carrier forms a cone with $\ell = \ell(x_0,\dots,x_q)$, which is in every set $U_{x_i}$.  Let $y$ be a point in $\scrC(x_0,\dots,x_q)$.  Either $y$ is one of $x_0,\dots,x_q$, or it was included as $\ell(\sigma)$ for some $\sigma\subset \{x_0,\dots,x_q\}$.  Because $d(\sigma) \le d(x_0,\dots,x_q)$, This means that $d(y,x_j) \le \frac{c}{c-1}r$ for any $x_j \in \sigma$.  We can then bound using triangle inequality
\begin{align}
    d(y,\ell) &\le d(y, x_j) + d(x_j,\ell)\\
    &\le \frac{c}{c-1} r + \frac{c}{c-1} r\\
    &\le \frac{2c}{c-1} r.
\end{align}
Because this holds for any point $y$ in the carrier, $\scrC(x_0,\dots,x_q)$ forms a cone by level $\frac{2c}{c-1} r$, so is acyclic.
\end{proof}
If we wish to obtain a $\alpha = 1+\epsilon$ interleaving, we can calculate that we must set $c = \frac{\epsilon+1}{\epsilon-1}$. In the limit of $c\to \infty$, $\epsilon\to 1$ from above, so we are limited to $\alpha > 2$ using this strategy.  We can achieve $\alpha = 3$ by setting $c = 3$, which limits the size of the sets in the cover while achieving a relatively small multiplicative interleaving bound.  A tighter bound can be achieved by re-weighting edges with distance $\ge \frac{c}{c-1} r$ in \cref{prop:sparse_carrier} -- see \cite{sheehyLinearSizeApproximationsVietoris2013,cavanna_geometric_2015} for details.

\section{Computations}\label{sec:computations}

In this section, we demonstrate the use of Vietoris-Rips cover complexes in studying the homology of point cloud data.  We first examine how several different covers can be used to investigate the homology of a sample from the torus.  Next, we use the greedy landmark cover of \cref{sec:sparse_filt_cover} to investigate the homology of $d$-dimensional Klein bottles associated with high-dimensional image patches. We have incorporated an implementation of the Vietoris-Rips cover complex into the BATS software\footnote{\url{https://github.com/CompTop/BATS}} package \cite{factorizationView2019} to support our experiments.

\subsection{A Flat Torus}\label{sec:computation_examples}

For our first example, we sample 500 points in a spiral on a flat torus in 4 dimensions.  For intermediate parameters of a filtration, we expect to generally see the homology of the torus $T^2$
\begin{equation}
    H_q(T^2) = \begin{cases}
    \FF & q = 0\\
    \FF \oplus \FF & q = 1\\
    \FF & q = 2
    \end{cases}
\end{equation}
with coefficients in any field.  

\begin{figure}
    \centering
    \includegraphics[width=0.49\linewidth]{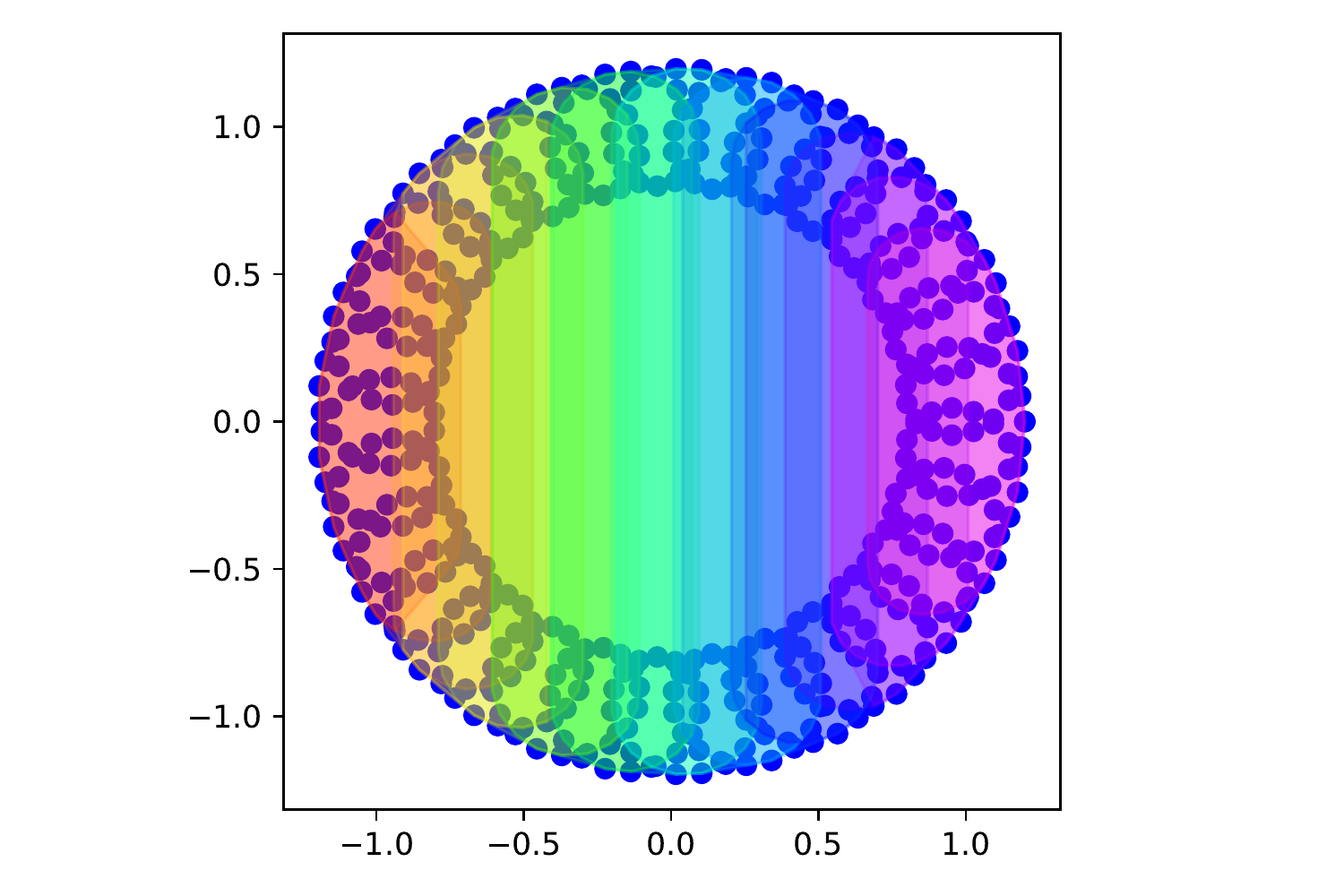}
    \includegraphics[width=0.49\linewidth]{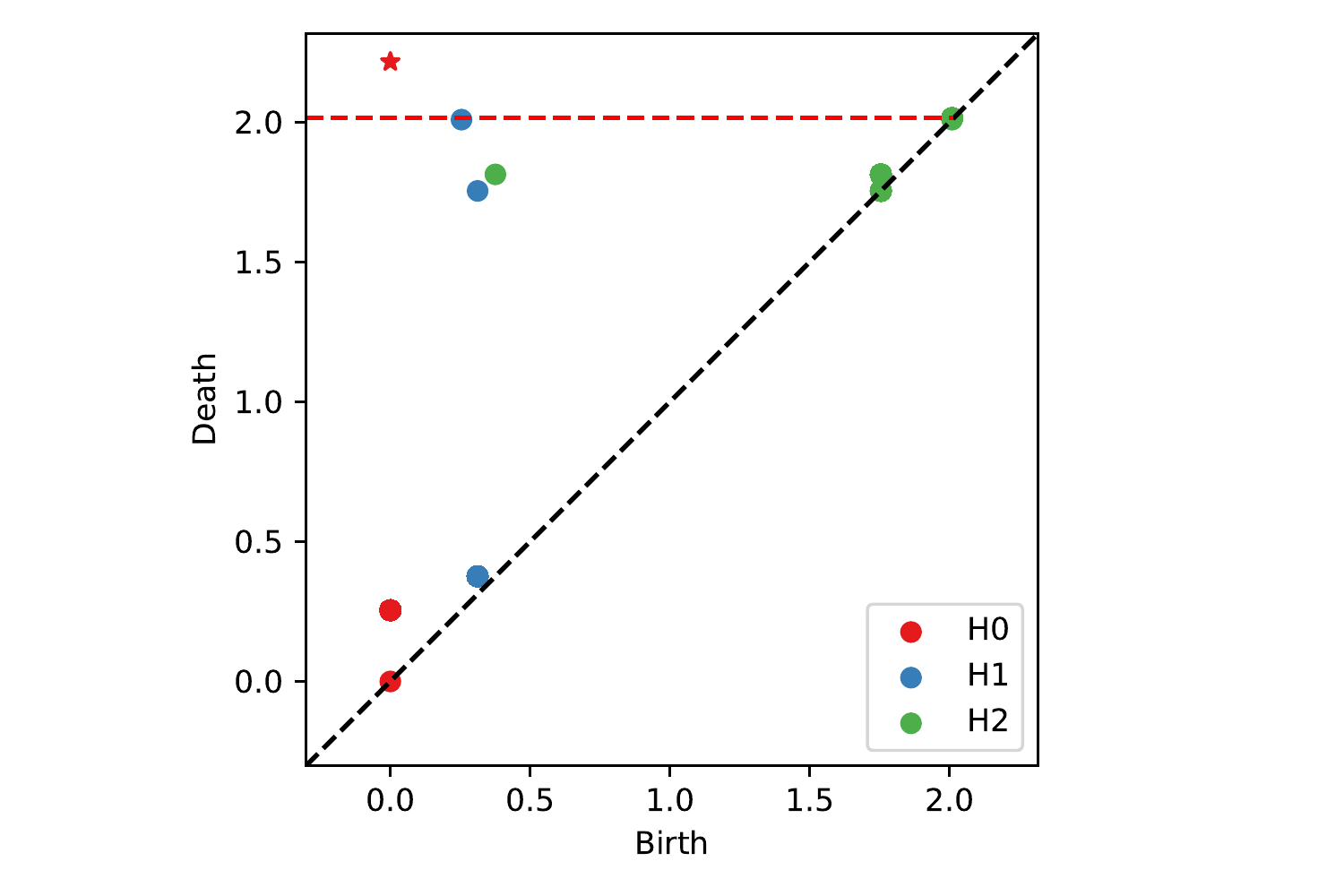}
    \caption{
    Cover of a flat torus pulled back from projection onto the first coordinate and the persistence diagram of $\calR(\bX, \calU; r)$.  The Nerve of the cover is contractible, as it covers an interval.  We see a single essential $H_0$ class (above the dashed red line), two persistent $H_1$ classes, and a persistent $H_2$ class.
    }
    \label{fig:torus_line_pullback}
\end{figure}

In \cref{fig:torus_line_pullback}, we pull back a cover of an interval in one dimension covering the projection of the data set onto the first coordinate.  In this case, each set in the cover has non-trivial structure - generally two robust connected components and two robust generators in $H_1$.  However, the persistent homology of the cover complex $\calR(\bX, \calU; r)$ demonstrates robust generators corresponding to the homology of the torus.

\begin{figure}
    \centering
    \includegraphics[width=0.49\linewidth]{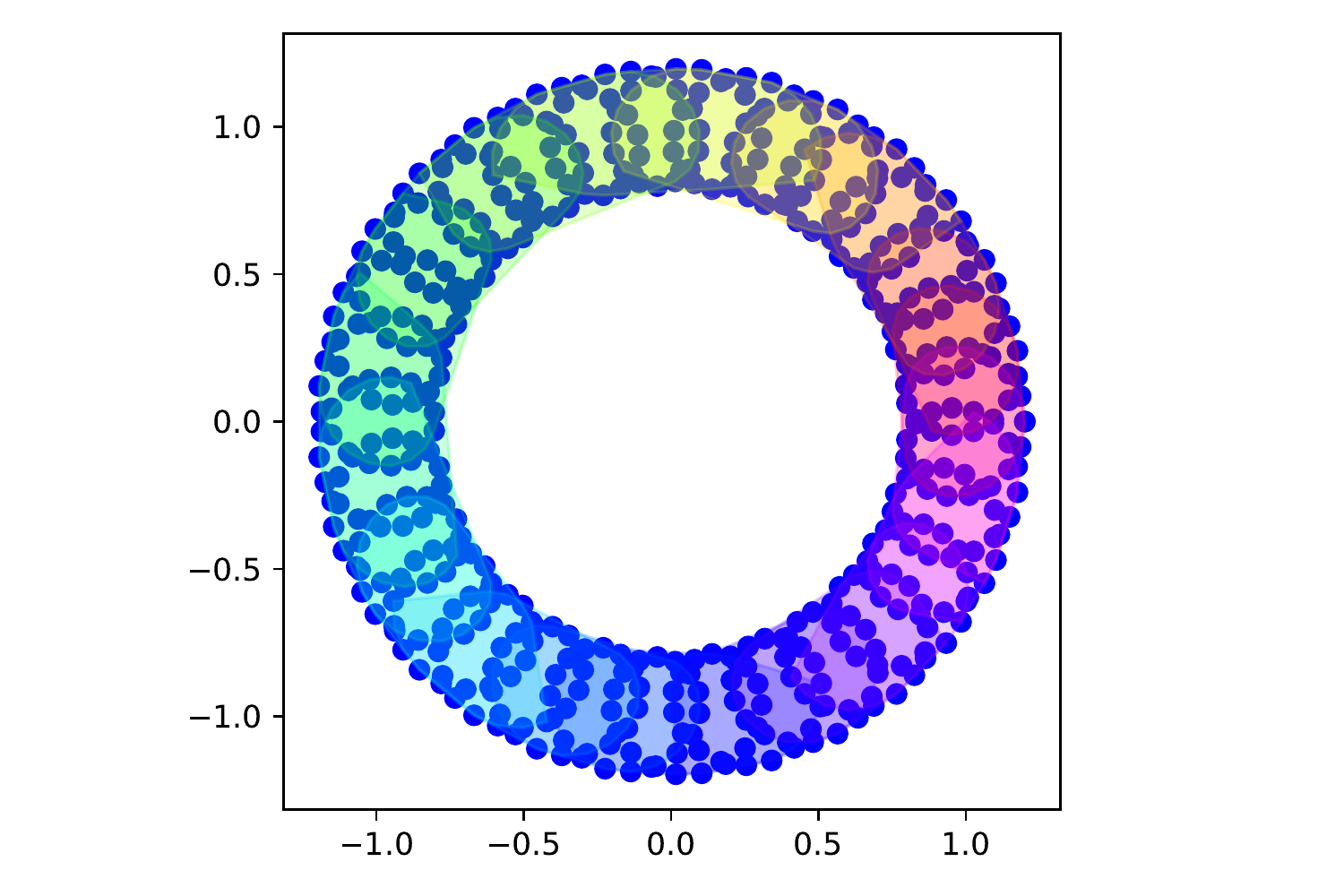}
    \includegraphics[width=0.49\linewidth]{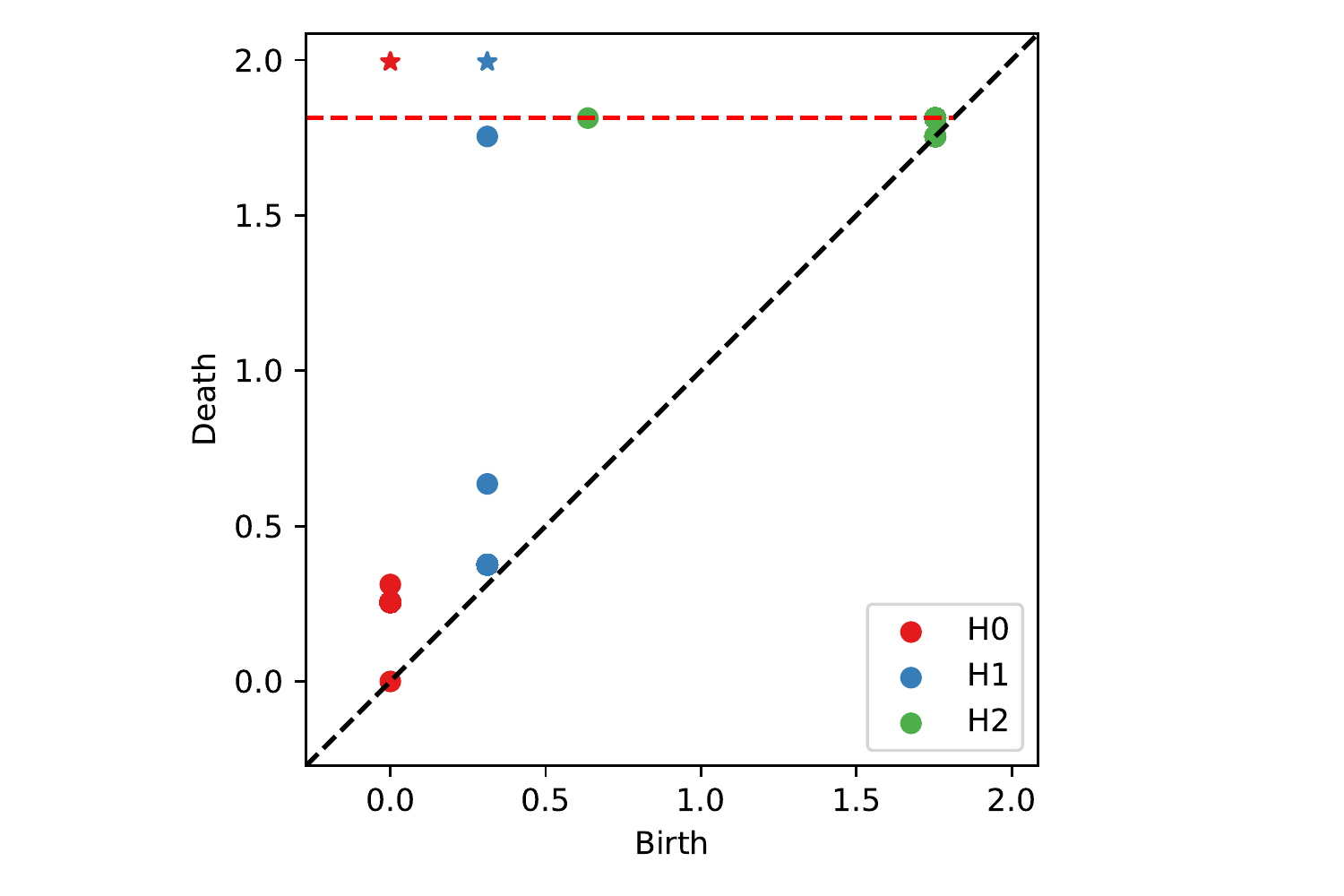}
    \caption{
    Cover of a flat torus pulled back from projection onto first two coordinates and the persistence diagram of $\calR(\bX, \calU; r)$.  The nerve of the cover is homotopic to the circle, and we see essential $H_0$ and essential $H_1$ classes from the nerve of this cover.  We also see an additional persistent $H_1$ class and persistent $H_2$ class.
    }
    \label{fig:torus_circle_pullback}
\end{figure}

In \cref{fig:torus_circle_pullback}, we pull back a cover of the the data set projected onto its first two coordinates.  In this case, the nerve of the cover is homotopic to a circle, and each set in the cover has points that lie in a circle.  In this case, the cover complex $\calR(\bX, \calU; r)$ has prominent homology classes for each class in the torus, but the the classes coming from the nerve of the cover are essential.

\begin{figure}
    \centering
    \includegraphics[width=0.49\linewidth]{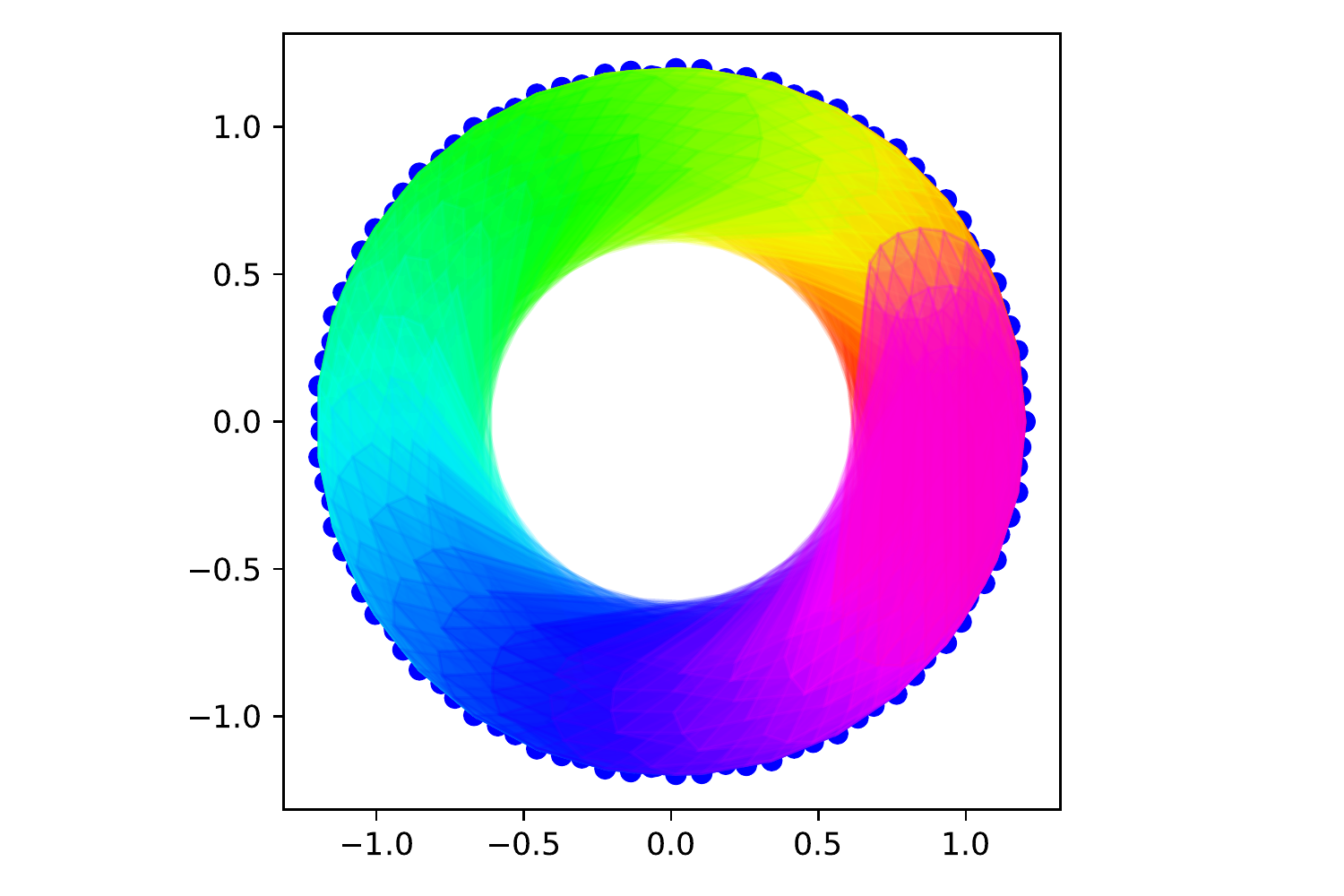}
    \includegraphics[width=0.49\linewidth]{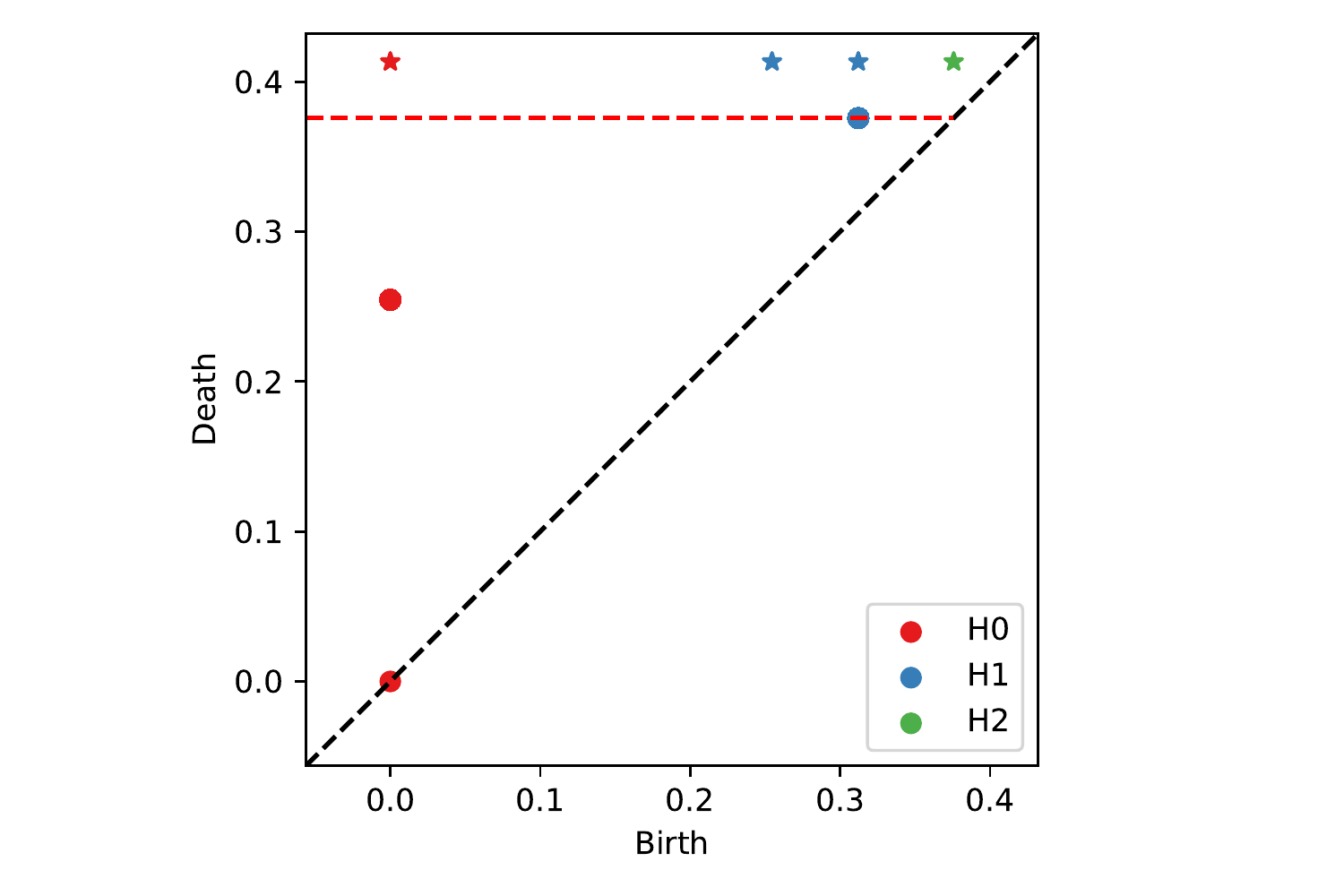}
    \caption{
    Cover of a flat torus obtained from the 20-nearest neighbors of each point and the persistence diagram of $\calR(\bX, \calU; r)$.  The nerve of the cover is homotopic to the torus, and we see essential homology classes corresponding to the homology of the torus.
    }
    \label{fig:torus_nn}
\end{figure}

In \cref{fig:torus_nn}, instead of a pullback cover we simply produce a cover containing a set for every point $ x\in \calX$ containing $x$ itself and its 20 nearest neighbors.  In this case, all sets are close to acyclic, but the nerve of the cover is equivalent to the torus, which we see reflected in the essential homology classes in the persistence diagram.

\begin{figure}
    \centering
    \includegraphics[width=0.49\linewidth]{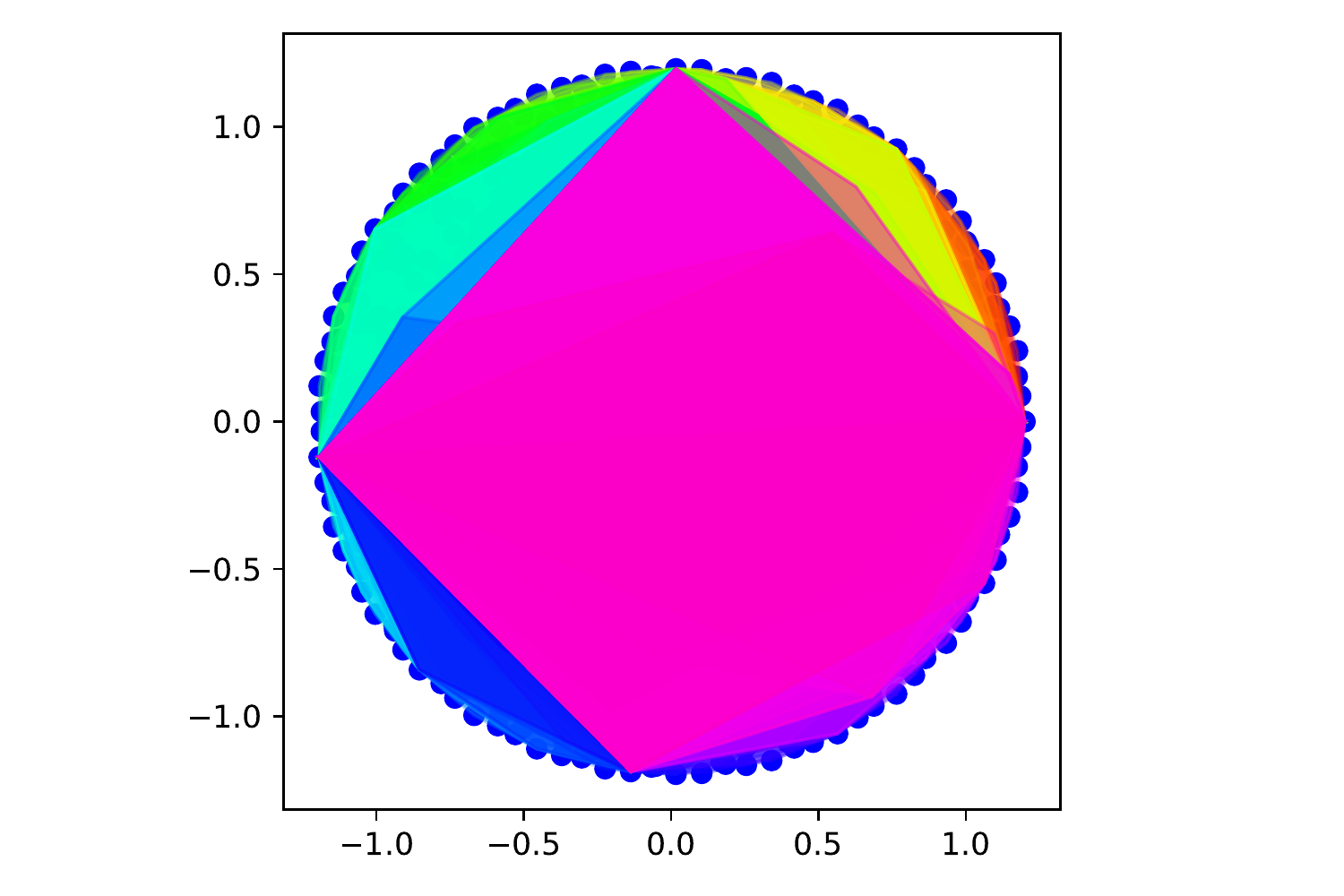}
    \includegraphics[width=0.49\linewidth]{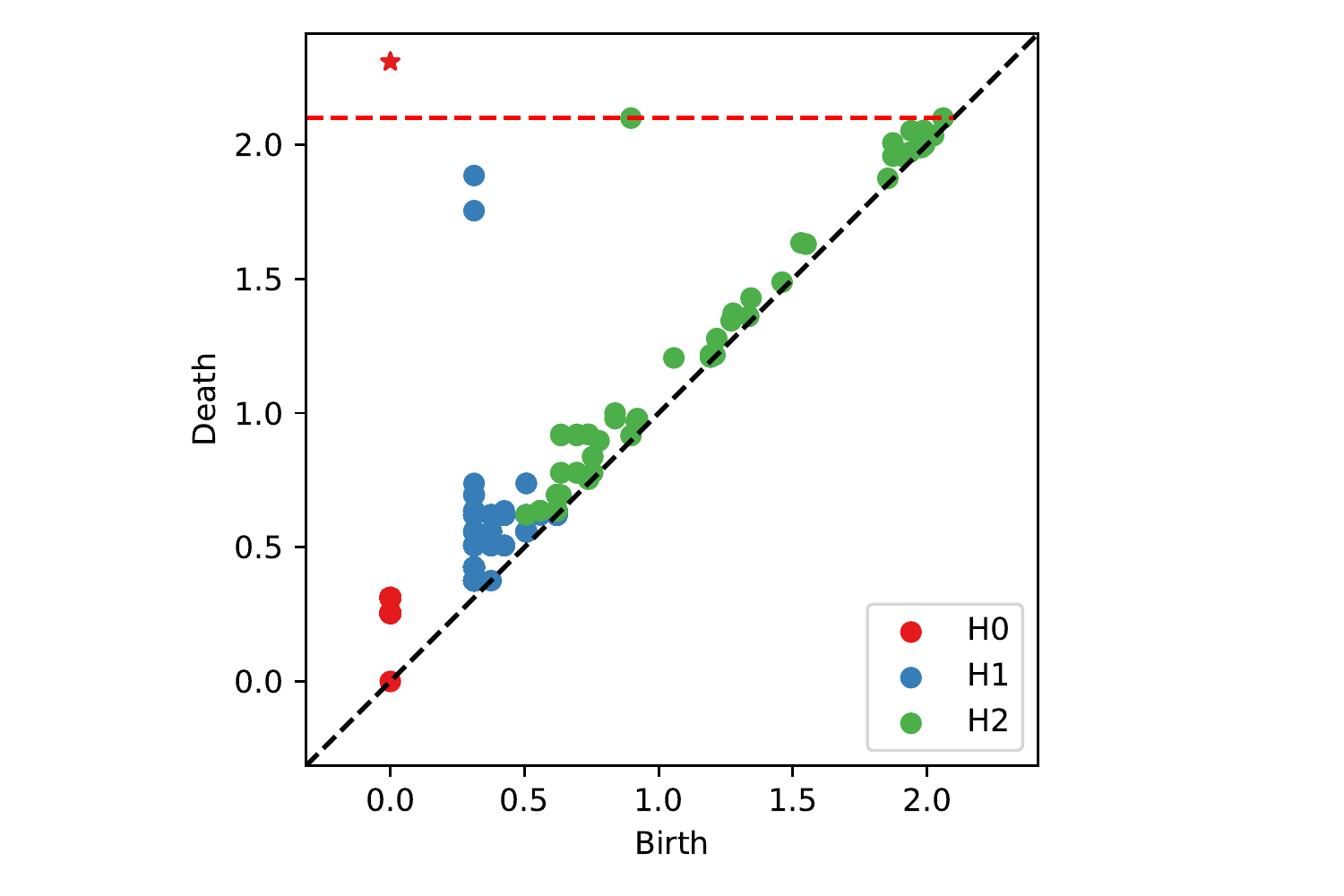}
    \caption{
    Cover of the flat torus based on the procedure in \cref{sec:sparse_filt_cover} and the persistence diagram of $\calR(\bX, \calU; r)$.  The Nerve is contractible, and so we just see a single essential $H_0$ class.  We also see two persistent $H_1$ classes and a persistent $H_2$ class.
    }
    \label{fig:torus_sparse}
\end{figure}

In \cref{fig:torus_sparse} we construct a cover of the data using the procedure in \cref{sec:sparse_filt_cover} with $c=1.0$ for maximum sparsity.  While this low value of $c$ gives a very pessimistic interleaving bound, each set in the cover is quite small, averaging less than $20$ points, and we still see the homology of the torus reflected in the prominent homology classes of the persistence diagram of $\calR(\bX, \calU; r).$





\subsection{$d$-Dimensional Klein bottles}

An interesting space motivated by data which admits a non-trivial fibration structure is a Klein bottle which lies near a high-density subset of high-contrast image patches \cite{CImgPatch}.  The fibration map can be obtained using the Harris edge detector \cite{Harris88,pereaTexture} which sends an image patch to the direction of largest variation.

In \cite{nelson_parameterized_2020}, this model is generalized to higher-dimensional images to obtain a fibration over $\RP{d-1}$ for $d$-dimensional images.  We will refer to this space as the $d$-dimensional Klein bottle, $\calK^d$, which was described independently in a different context by Davis \cite{davis_n-dimensional_2019}.  The homology of this space can be computed using the Leray-Serre spectral sequence \cite{McClearySS} -- see \cite{nelson_parameterized_2020} for explicit computational details.
\begin{equation}\label{eq:harris_homology}
    H_k(\calK^d) = \begin{cases}
        \ZZ & k = 0\\
        \ZZ_2 \oplus \ZZ_2 & 0 < k < d-1,~ k~\text{odd}\\
        \ZZ & k = d,~ d~\text{odd}\\
        \ZZ \oplus \ZZ_2 & k = d-1,~ d~\text{even}\\
        0 &\text{otherwise}
    \end{cases}
\end{equation}

Using the universal coefficient theorem (c.f. \cite{HatcherAT}, 3A.3), we see different dimensions in homology when computing with fields of different characteristic due to the 2-torsion in the integral homology of $\calK^d$ 
\begin{equation}\label{eq:f2coeff}
H_k(\calK^d; \FF_2) = \begin{cases}
    \FF_2 & k = 0, d~\text{odd}\\
    \FF_2 \oplus \FF_2 & 0 < k < d\\
    \FF_2 & k = d\\
    0 &\text{otherwise}
\end{cases}
\end{equation}
and for $\FF = \FF_p$, $p> 2$, or $\FF = \QQ$, we have
\begin{equation}\label{eq:f3coeff}
H_k(\calK^d; \FF) = \begin{cases}
    \FF & k = 0\\
    \FF & k = d-1, d~\text{even}\\
    \FF & k = d, d~\text{odd}\\
    0 &\text{otherwise}
\end{cases}
\end{equation}

\begin{figure}
    \centering
    \includegraphics[width=0.49\linewidth]{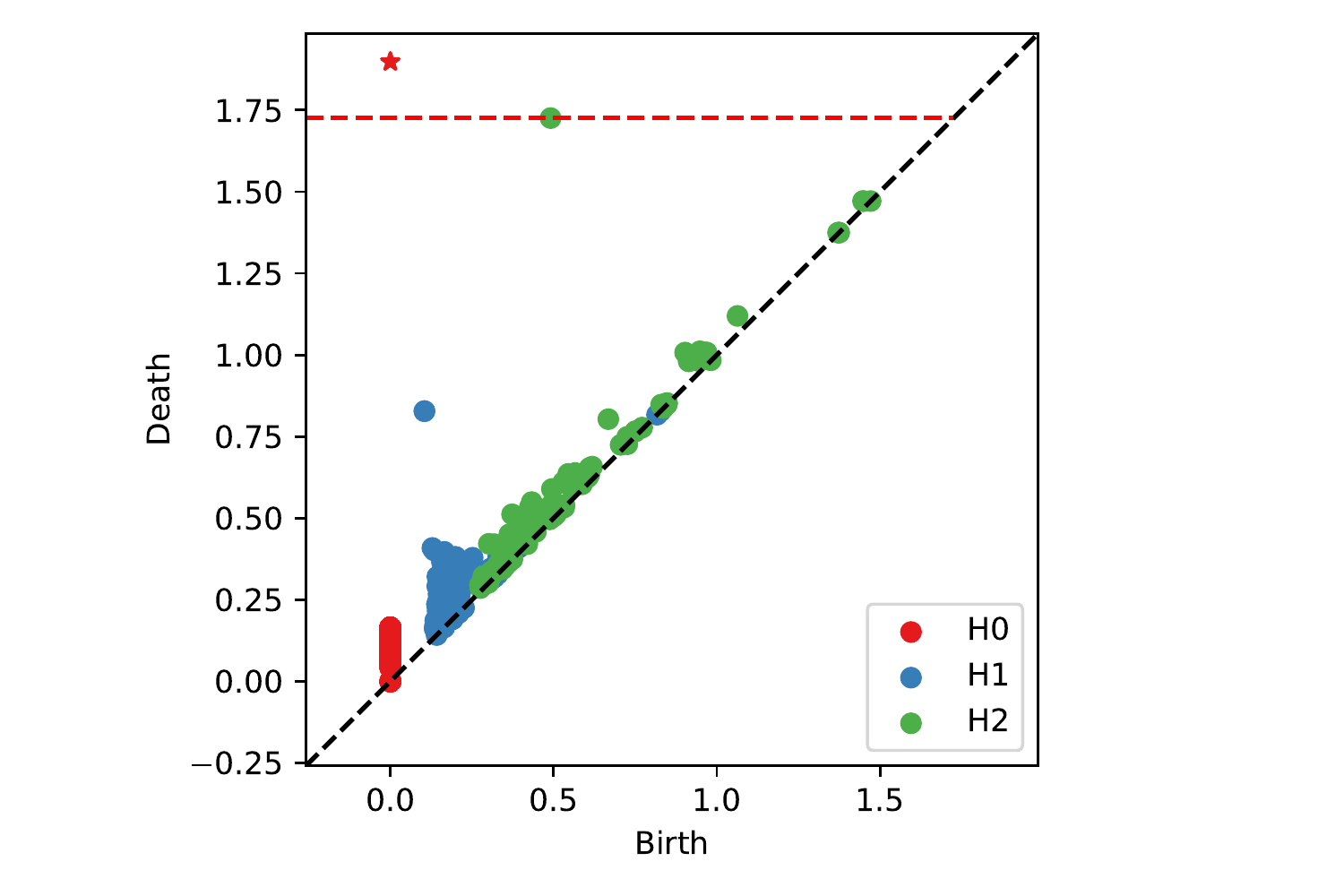}
    \includegraphics[width=0.49\linewidth]{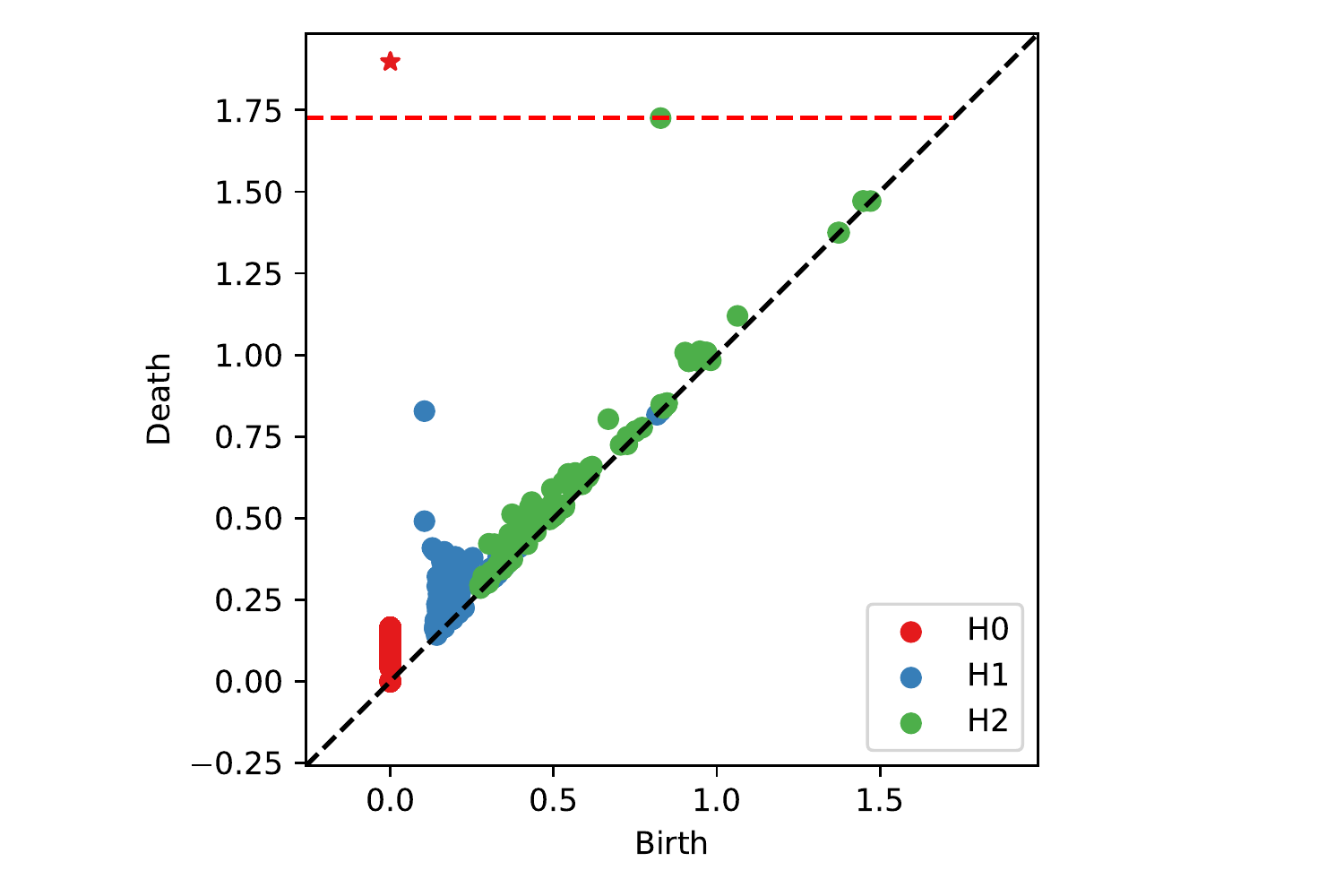}
    \caption{
    Persistent homology of a 2-dimensional Klein bottle, $\calK^2$. Left: with $\FF_2$ field coefficients.  Right: with $\FF_3$ field coefficients.  There are two robust $H_1$ generators with $\FF_2$ coefficients at the location $(0.1,0.8)$.
    }
    \label{fig:k2}
\end{figure}

\begin{figure}
    \centering
    \includegraphics[width=0.49\linewidth]{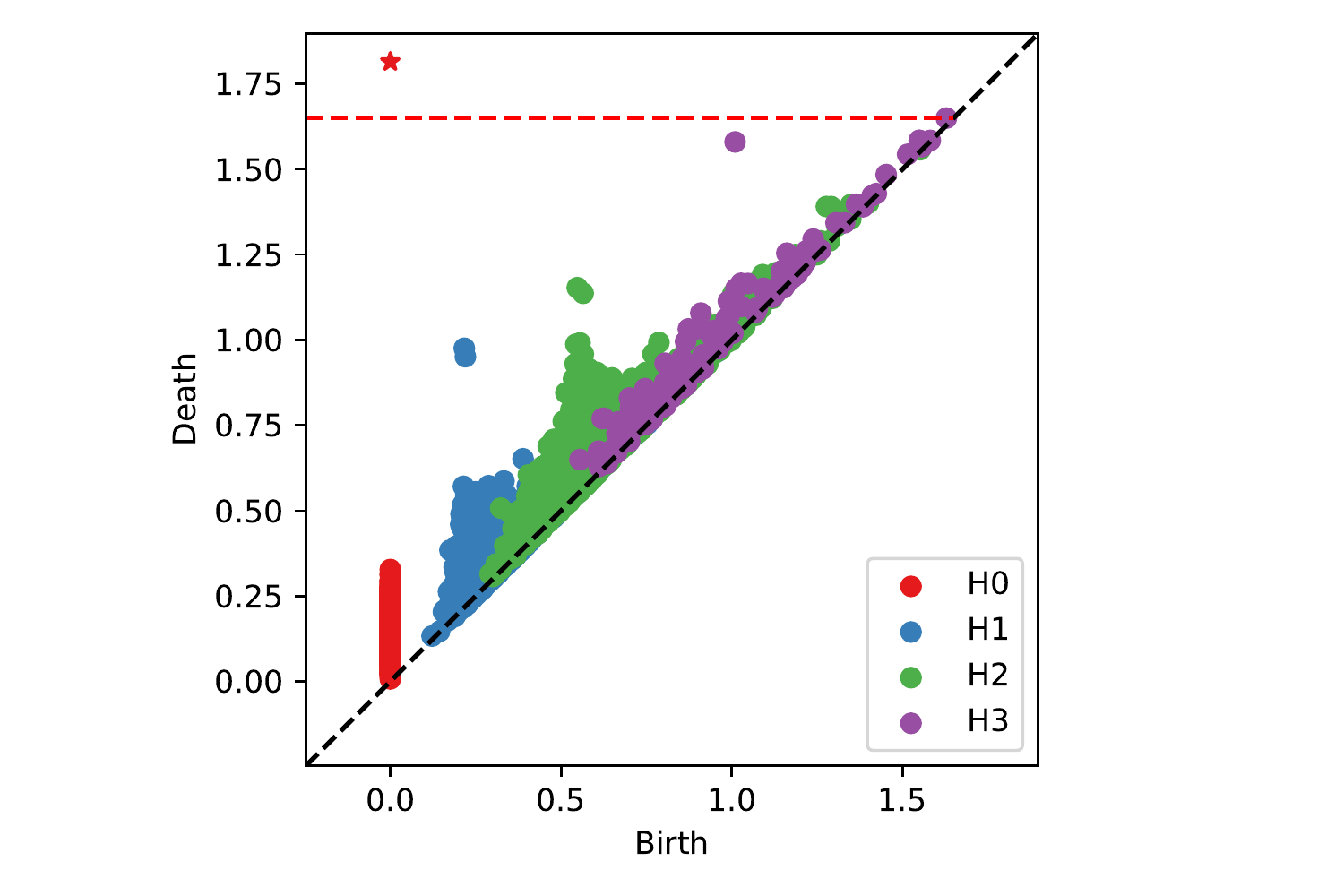}
    \includegraphics[width=0.49\linewidth]{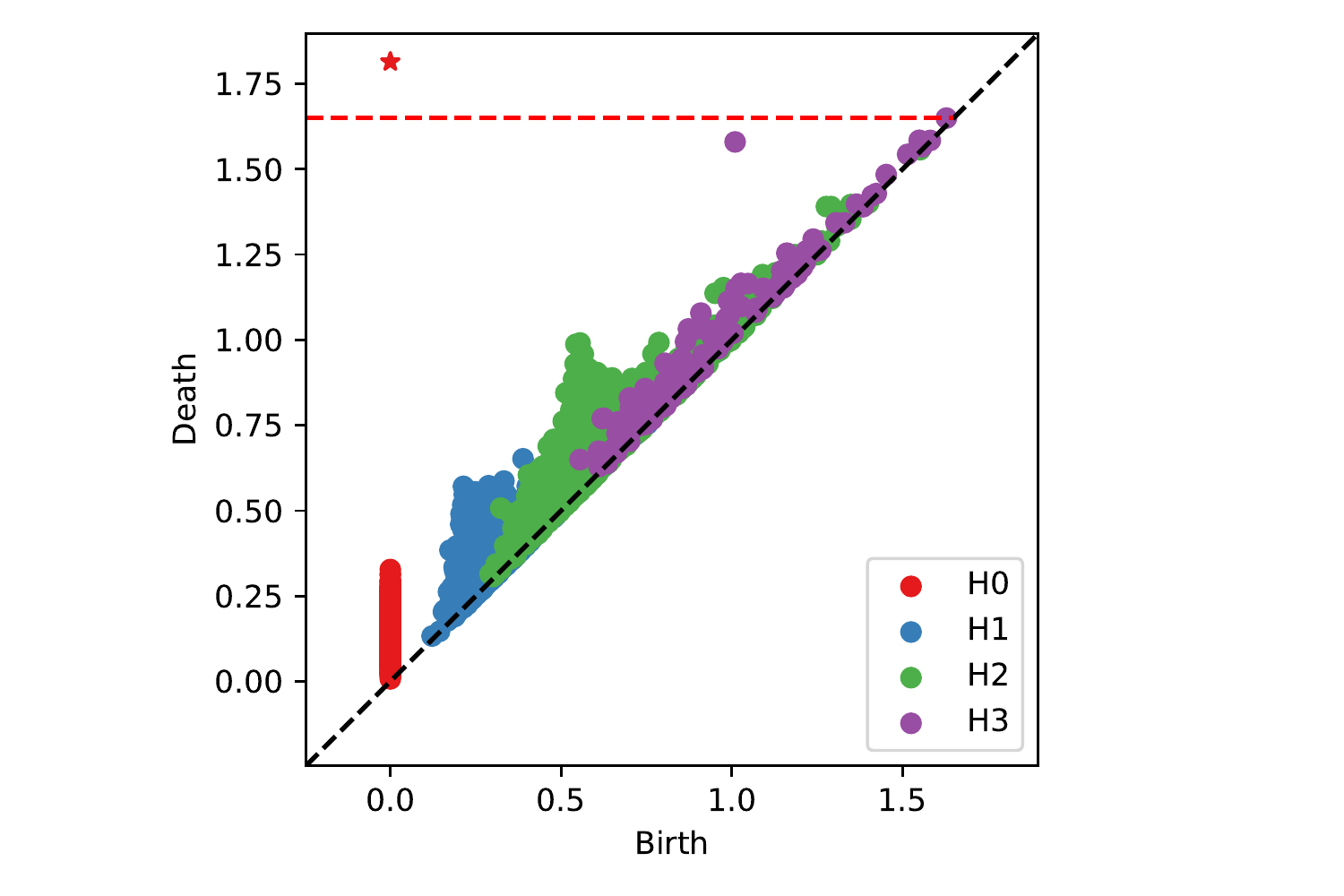}
    \caption{
    Persistent homology of a 3-dimensional Klein bottle, $\calK^3$. Left: with $\FF_2$ field coefficients.  Right: with $\FF_3$ field coefficients.
    }
    \label{fig:k3}
\end{figure}

We obtain a sample of $\K^d$ generalizing the model of Carlsson et al \cite{CImgPatch}.  Given a unit vector $\vphi \in \RR^d$ and an angle $\theta$, we define a patch
\begin{equation}
   p(x; \theta, \vphi) = \cos(\theta) (x^T \vphi)^2 + \sin(\theta) (x^T \vphi)
\end{equation}
which can be evaluated as a pixelated image patch by evaluating $x\in \RR^d$ on a grid.  In \cref{fig:k2} we generate a Klein bottle $\calK^2$ on $3\times 3$ image patches by evaluating $x$ on the grid $\{-1,0,1\}^2$, 20 equally spaced values of $\theta$, and $50$ equally spaced values of $\vphi$ for a total of 1000 points in $\RR^9$.  We compute persistent homology of the cover complex filtration $\calR(\bX, \calU; r)$ using the landmark-based cover in \cref{sec:sparse_filt_cover}, with $c=1.0$ for maximum sparsity. Using \cref{eq:f2coeff}, the homology of $\calK^2$ with coefficients in $\FF^2$ has has dimension vector $(1,2,1)$, which is clearly observed in the persistence diagram.  In \cref{eq:f3coeff}, coefficients in $\FF^3$, the dimension vector becomes $(1,1,0)$, and we see one of the prominent $H_1$ classes shrink, and the prominent $H_2$ class shift toward the diagonal in the corresponding persistence diagram.

In \cref{fig:k3}, we generate $\calK^3$ on $5\times 5\times 5$ patches using $20$ equally-spaced values of $\theta$ and $150$ values of $\vphi$ chosen by greedily landmarking a larger set of $4000$.  The total data set consists of 3000 points in $\RR^{125}$, and again we compute persistent homology of the cover complex filtration $\calR(\bX, \calU; r)$ using the landmark-based cover in \cref{sec:sparse_filt_cover}, with $c=1.0$.  Using \cref{eq:f2coeff}, the homology dimension vector of this space in $\FF^2$ is (1,2,2,1), and for $\FF^3$ coefficents, the dimension vector is $(1,0,0,1)$.  In \cref{fig:k3} we see both these dimension vectors match with the prominent homology classes in each dimension.



\section{Conclusion}\label{sec:conclusion}

In this paper, we developed a filtered version of the acyclic carrier theorem, which allowed us to construct interleavings between different geometric constructions.  We have presented several results for Vietoris-Rips cover complexes, and we anticipate that the use of filtered carriers has broad potential as a technique to construct interleavings in situations that we have not yet considered.  We have focused on algebraic interleavings, and many of these results could potentially be extended to homotopy interleavings \cite{blumbergUniversalityHomotopyInterleaving2017} given additional care when constructing carriers of cell complexes.  Another interesting line of future investigation would be to use the algorithmic construction of maps from carriers 
in data analysis.  This could potentially be used, for instance, in constructing low dimensional embeddings of data that minimize the interleaving distance between a filtration on the higher-dimensional point cloud and the embedded point cloud.

Another line of future work is to leverage cover complexes for distrubyted computation.  A limited version of this was explored in \cite{yoon2018}, and our interleaving results expand the potential use of cover complexes to more general settings.  We also believe that the interleaving bounds we derive are likely pessimistic in many situations where data has additional structure.  Analyses of these situations may help tighten our bounds considerably.

\section*{Acknowledgements}
BJN was supported by the Defense Advanced Research Projects Agency (DARPA) under Agreement No.
HR00112190040.  He thanks Gunnar Carlsson and Jonathan Taylor for discussions on a early version of this work.













\bibliographystyle{spmpsci}
\bibliography{main}
%
%

\end{document}